\documentclass[5p,twocolumn]{elsarticle}
\usepackage{amsmath}
\usepackage{amssymb}
\usepackage{amsfonts}
\usepackage{mathtools}
\usepackage{forest}
\newcommand*{\id}{{\mathrm{id}}}

\newcommand*{\la}{{\langle}}                                                  
\newcommand*{\ra}{{\rangle}}                                                  
                                                        
\newcommand*{\R}{{\mathbb R}}

\newcommand*{\Hb}{{\mathbb H}}                                                
                                                
\newcommand*{\Vb}{{\mathbb V}}
                                                
\newcommand*{\Sb}{{\mathbb S}}

\newcommand*{\Qb}{{\mathbb Q}}

\newcommand*{\pa}{{\partial}}

\newcommand*{\rd}{{\mathrm d}}

\newcommand*{\gf}{{\mathfrak g}}  
 
\newcommand*{\pf}{{\mathfrak p}}

\newcommand*{\ff}{{\mathfrak f}}
\newcommand*{\hf}{{\mathfrak h}}

\newcommand*{\lb}{\langle}
\newcommand*{\rb}{\rangle}

\newtheorem{theorem}{Theorem}
\newtheorem{lemma}[theorem]{Lemma}

\newdefinition{definition}{Definition}
\newdefinition{prescription}{Prescription}
\newdefinition{remark}{Remark}
\newdefinition{example}{Example}
\newdefinition{procedure}{Procedure}
\newproof{proof}{Proof}

\journal{...}

\begin{document}
\begin{frontmatter}
\title{Coagulation, non-associative algebras and binary trees} 
\author{Simon~J.A.~Malham\fnref{fn1}}    
\ead{S.J.A.Malham@hw.ac.uk} 
\fntext[fn1]{SJAM was supported by an EPSRC Mathematical Sciences Small Grant EP/X018784/1}
\affiliation{organisation={Maxwell Institute for Mathematical Sciences,        
and School of Mathematical and Computer Sciences},   
addressline={Heriot-Watt University}, postcode={EH14~4AS}, city={Edinburgh}, country={UK}}

\begin{abstract}
  We consider the classical Smoluchowski coagulation equation with a general frequency kernel.
  We show that there exists a natural deterministic solution expansion in the non-associative algebra generated
  by the convolution product of the coalescence term. The non-associative solution expansion
  is equivalently represented by binary trees. We demonstrate that the existence of such solutions
  corresponds to establishing the compatibility of two binary-tree generating procedures, by:
  (i) grafting together the roots of all pairs of order-compatibile trees at preceding orders, or
  (ii) attaching binary branches to all free branches of trees at the previous order.
  We then show that the solution represents a linearised flow, and also establish
  a new numerical simulation method based on truncation of the solution tree expansion
  and approximating the integral terms at each order by fast Fourier transform. 
  In particular, for general separable frequency kernels, the complexity of the method
  is linear-loglinear in the number of spatial modes/nodes. 
\end{abstract}
\end{frontmatter}

\section{Introduction}\label{sec:intro}
Herein we consider the classical Smoluchowski coagulation equation with a general frequency kernel.
We show there exists a natural deterministic solution expansion in a non-associative algebra equivalent to an
expansion on binary trees. The algebra product is generated by the coagulation convolution term
and is non-associative for general frequency kernels. We demonstrate how the existence of
solutions to the coagulation equation is equivalent to the compatibility                              
of generating binary trees at each order, by the two following natural procedures, by:
(i) grafting together the roots of all pairs of order-compatible trees at preceding orders, and
(ii) attaching binary branches to all the possible free branches of trees at the previous order.
We examine the classical constant, additive and multiplicative frequency kernels as special cases within the binary tree context.
Lastly, we also establish that the flow has a linearised flow formulation,
and outline how the flow can be approximated to develop new numerical simulation methods.

The Smoluchowski coagulation equation has the form,
\begin{align}\label{eq:Smoluchowski}
  \pa_tg(x;t)=&\;\tfrac12\int_0^x K(x-y,y)g(x-y;t)g(y;t)\,\rd y\nonumber\\
  &\;-g(x;t)\int_0^\infty K(x,y)g(y;t)\,\rd y,
\end{align}
where $g=g(x;t)$ denotes the density of molecular clusters of mass $x$, and $K=K(x,y)$ is the given symmetric frequency kernel.
We assume the initial data is $g_0$ so $g(x;0)=g_0(x)$. All the variables concerned here are positive.
The exact form of the frequency kernel depends on the application considered.
Three special cases are often the focus of investigation, namely the constant, $K=1$, additive, $K=x+y$, and multiplicative, $K=xy$, kernel cases.
This is because the Smoluchowski equation can be solved explicitly for these cases. See Aldous~\cite{Aldous} and Menon and Pego~\cite{MP} for more details.
In the multiplicative kernel case a phase transition occurs at a gelation time, beyond which the solution can be extended, see Leyvraz and Tschudi~\cite{LT}. 
Herein, we focus on the general $K=K(x,y)$ kernel case. It is usual to consider 
the Laplace or Bernstein transform of the Smoluchowski's equation; see Menon and Pego~\cite{MP}.
For the moment let us consider a general linear transformation $\mathcal H$ analogous to such transforms.
We define the general linear transformation of a function $g=g(x)$ with support on $[0,\infty)$ by, 
\begin{equation}\label{eq:generaltransform}
\bigl(\mathcal Hg\bigr)(s;t)\coloneqq\int_0^\infty h(s,x)\,g(x)\,\rd x,
\end{equation}
where $0\leqslant h\leqslant 1$. If $h\coloneqq\mathrm{e}^{-sx}$ then $\mathcal H$ is the Laplace transform.
If $h(s,x)\coloneqq1-\mathrm{e}^{-sx}$, then $\mathcal H$ is the Bernstein transform. 
It is straightforward to show that if we consider the general transform of the Smoluchowski equation~\eqref{eq:Smoluchowski}
and $\gf\coloneqq\mathcal H\,g$, then $\gf=\gf(s;t)$ satisfies, 
\begin{equation}\label{eq:evol}
  \pa_t\gf(s;t)=\iint g(y;t)\bigl(H(s,y,z)K(y,z)\bigr)g(z;t)\,\rd y\,\rd z,
\end{equation}
where the double integral shown is over $[0,\infty)^2$, and we set $H(s,y,z)\coloneqq\frac12(h(s,y+z)-h(s,y)-h(s,z))$.
If we set $g=\mathcal H^{-1}\gf$ in both places where $g$ appears on the right in~\eqref{eq:evol},
the resulting vector field involves a quadruple integral of a quadratic form involving $\gf$ with a weight $H(s,y,z)K(y,z)v(y,\cdot)v(z,\cdot)$.
Here $v$ is the kernel of the operator $\mathcal H^{-1}$. So, for example, in the instance of the Laplace transform $v(s,x)=\mathrm{e}^{xs}$,
and the quadruple integral additionally involves two Bromwich contour integrals. The resulting evolution equation 
for $\gf$ has the form, 
\begin{equation}\label{eq:quadnonassoc}
\pa_t\gf=\gf\star\gf, 
\end{equation} 
where the product `$\star$' is based on the quadratic form of the vector field in the
evolution equation~\eqref{eq:evol} with the quadruple integral. There are many ways to
represent Smoluchowski's coagulation equation as an abstract evolution equation of the form~\eqref{eq:quadnonassoc}
with the quadratic vector field shown. 
We consider another analogous representation for Smoluchowski's coagulation equation in Section~\ref{sec:coagulation}.
In all of these representations which have the abstract form~\eqref{eq:quadnonassoc}, the product `$\star$' is in general non-associative.
That this is the case is straightforwardly checked. The constant frequency kernel case is the exception, the product `$\star$' is associative in
this singular instance. Naturally the form of $\mathfrak f\star\gf$ for general transforms of any pair of functions $f$ and $g$
with support on $[0,\infty)$, for any of the representations mentioned, is straightforwardly implied.

Hence, with `$\star$' a non-associative product, our goal is to solve the abstract equation~\eqref{eq:quadnonassoc}, or equivalently,
\begin{equation}\label{eq:abstract}
\gf(t)=\xi+\int_0^t\gf(t^\prime)\star\gf(t^\prime)\,\rd t^\prime, 
\end{equation} 
for $\gf=\gf(s;t)$, where $\xi=\xi(s)$ is the general transform of the data $g_0=g_0(x)$. 
The solution can be formally derived by iteration. This generates the solution expansion,
\begin{align}
  \gf=&\;\xi+t\,\xi\star\xi+\tfrac12t^2\bigl(\xi\star(\xi\star\xi)+(\xi\star\xi)\star\xi\bigr)\nonumber\\
  &\;+\tfrac16t^3\Bigl(\xi\star\bigl(\xi\star(\xi\star\xi)\bigr)+\xi\star\bigl((\xi\star\xi)\star\xi\bigr)\nonumber\\
  &\;\qquad+2\cdot(\xi\star\xi)\star(\xi\star\xi)+\bigl(\xi\star(\xi\star\xi)\bigr)\star\xi\nonumber\\
  &\;\qquad+\bigl((\xi\star\xi)\star\xi\bigr)\star\xi\Bigr)+\mathcal{O}(t^4). \label{eq:formalsoln}
\end{align}
In the context of Smoluchowski's equation, the product `$\star$', though non-associative, is commutative.
This means that we can simplify the expansion by combining some like terms. For example the terms
of order $t^2$ combine, as do all the terms of order $t^3$ except the symmetric term with the real factor `$2$'.
We discuss such symmetries in detail in Section~\ref{sec:coagulation}. However, in general we assume the product `$\star$'
is non-commutative as we can complete all our analysis in this more general context, and specialise to
the commutative context as and when we require. Keeping track of the bracketed terms due to the non-associativity
at higher order is most easily accomplished using the representation of the terms in the expansion by rooted planar binary trees.
Besides the empty tree $\emptyset$, the first set of rooted planar binary trees, up to and including grade $3$, are,\medskip

{\tiny \begin{forest} for tree={grow'=90, parent anchor=center,child anchor=center} [] \path[fill=black] (.anchor) circle[radius=1.5pt]; \end{forest}},
{\tiny \begin{forest} for tree={grow'=90, parent anchor=center,child anchor=center, l=0cm,inner ysep=0cm,edge+=thick} [[][]] \end{forest}},
{\tiny \begin{forest} for tree={grow'=90, parent anchor=center,child anchor=center, l=0cm,inner ysep=0cm,edge+=thick} [[][[][]]] \end{forest}},
{\tiny \begin{forest} for tree={grow'=90, parent anchor=center,child anchor=center, l=0cm,inner ysep=0cm,edge+=thick} [[[][]][]] \end{forest}},
{\tiny \begin{forest} for tree={grow'=90, parent anchor=center,child anchor=center, l=0cm,inner ysep=0cm,edge+=thick} [[][[][[][]]]] \end{forest}},
{\tiny \begin{forest} for tree={grow'=90, parent anchor=center,child anchor=center, l=0cm,inner ysep=0cm,edge+=thick} [[][[[][]][]]] \end{forest}},
{\tiny \begin{forest} for tree={grow'=90, parent anchor=center,child anchor=center, l=0cm,inner ysep=0cm,edge+=thick} [[[][]][[][]]] \end{forest}},
{\tiny \begin{forest} for tree={grow'=90, parent anchor=center,child anchor=center, l=0cm,inner ysep=0cm,edge+=thick} [[[][[][]]][]] \end{forest}},
{\tiny \begin{forest} for tree={grow'=90, parent anchor=center,child anchor=center, l=0cm,inner ysep=0cm,edge+=thick} [[[[][]][]][]] \end{forest}}.\medskip

\noindent Respectively, in order, these trees encode the successive terms in our formal solution expansion~\eqref{eq:formalsoln} for $\gf$.
Each vertex indicates a `$\star$' product while each free branch indicates the instance of a $\xi$ factor. Hence the first term `$\xi$' in~\eqref{eq:formalsoln}
is represented by the `{\tiny \begin{forest} for tree={grow'=90, parent anchor=center,child anchor=center} [] \path[fill=black] (.anchor) circle[radius=1.5pt]; \end{forest}}',
which, say, we express as $\xi$({\tiny \begin{forest} for tree={grow'=90, parent anchor=center,child anchor=center} [] \path[fill=black] (.anchor) circle[radius=1.5pt]; \end{forest}}).
The second term `$\xi\star\xi$' is represented by 
`\raisebox{-1pt}{{\tiny \begin{forest} for tree={grow'=90, parent anchor=center,child anchor=center, l=0cm,inner ysep=0cm,edge+=thick} [[][]] \end{forest}}}'
and we express this using the shorthand notation
$\xi$(\raisebox{-1pt}{{\tiny \begin{forest} for tree={grow'=90, parent anchor=center,child anchor=center, l=0cm,inner ysep=0cm,edge+=thick} [[][]] \end{forest}}}),
and so forth. Indeed a representation for $\gf$ in \eqref{eq:formalsoln} is thus,
\begin{equation*}
  \gf=\xi({\tiny \begin{forest} for tree={grow'=90, parent anchor=center,child anchor=center} [] \path[fill=black] (.anchor) circle[radius=1.5pt]; \end{forest}})
  +t\,\xi(\raisebox{-1pt}{{\tiny \begin{forest} for tree={grow'=90, parent anchor=center,child anchor=center, l=0cm,inner ysep=0cm,edge+=thick} [[][]] \end{forest}}})
  +\tfrac12t^2\biggl(\xi\Bigl(\raisebox{-4pt}{{\tiny \begin{forest} for tree={grow'=90, parent anchor=center,child anchor=center, l=0cm,inner ysep=0cm,edge+=thick} [[][[][]]] \end{forest}}}\Bigr)
  +\xi\Bigl(\raisebox{-4pt}{{\tiny \begin{forest} for tree={grow'=90, parent anchor=center,child anchor=center, l=0cm,inner ysep=0cm,edge+=thick} [[[][]][]] \end{forest}}}\Bigr)\biggr)
+\mathcal{O}(t^3). 
\end{equation*}
We observe, if we drop the $\xi$'s, we can represent this solution expansion for $\gf$ as an expansion in the algebra of planar binary trees over the field $\R$.

The purpose of this paper is to:
\begin{enumerate}
\item[(i)] Prove the full tree expansion \eqref{eq:formalsoln} is the solution to the abstract non-associative and non-commutative quadratic differential equation
  \eqref{eq:quadnonassoc} with data $\xi$. The proof is equivalent to establishing the compatability, with mutliplicity, of the grafting and branching 
  tree generating procedures;  
\item[(ii)] Give closed form expressions for the solution;
\item[(iii)] Show that the solution can be expressed as a linearised flow;
\item[(iv)] Establish a practical numerical simulation method for the Smoluchowksi coagulation equation, with a general frequency kernel,
   to evaluate the solution at any time $t$, up to gelation. The method is particularly efficient for separable frequency kernels.
\end{enumerate}
There are many comprehensive surveys on Smoluchowksi coagulation, more recent examples include Aldous~\cite{Aldous}, giving a general overview,
da Costa~\cite{daC}, focusing on deterministic aspects and Hammond~\cite{Hammond}, focusing on stochastic aspects. Conditions for well-posedness
of the Smoluchowski equation, the gelation time, as well as the phenomenon of instantaneous gelation can be found therein, as well as in Dubovskii~\cite{Dubovskii}.
The work of Menon and Pego~\cite{MP} and Iyer~\textit{et al.\/} \cite{ILP} in particular, establishes a general rigorous foundation for
the Smoluchowski equation and its extension to multiple mergers.
On the theory of differential equations in the context of non-associative algebras, early work includes that by Markus~\cite{Markus},
as well as R\"ohrl~\cite{Rohrl}.

The algebra of rooted planar trees is exceedingly rich and has an illustrious history dating back to Cayley. More recently, in the work of Butcher~\cite{Butcher},
they arose in the context of Runge--Kutta methods in numerical analysis in the form of the Butcher group and in the work of Connes and Kreimer~\cite{CK}
on perturbative quantum field theory in the form of a Hopf algebra of trees. These two objects are equivalent, as demonsrtated by Brouder~\cite{Brouder}. 
A Hopf algebra is an algebra endowed with a compatabile co-product, and an antipode, see Foissy~\cite{Foissy} for an introduction.
In the algebra context, the Hopf algebra of planar binary trees was established by Loday and Ronco~\cite{LR}. 
Also see Hairer \textit{et al.\/} \cite{HLW} for a detailed account of the use of trees in numerical analysis.
The connection between Smoluchowski coagulation and binary trees is not new, and indeed binary trees appear, for example,
in Aldous~\cite{Aldous} in the context of Galton--Watson processes, in the combinatorial approach in Spouge~\cite{Spouge},
and in Sheth and Pitman~\cite{ShethPitman} as merger history trees. Also see Marckert and Wang~\cite{MarW}.

Smoluchowski's coagulation equation has an extremely wide range of application. Classically, it is a model for 
polymerisation, aerosols, clouds/smog, clustering of stars and galaxies, schooling and flocking;
see Aldous~\cite{Aldous}. Recently it has been used as a model for genealogy, see Lambert and Schertzer~\cite{LS};
coarsening, see Gallay and Mielke~\cite{GM}; nanostructures on substrates such as ripening or `island coarsening', see Stoldt \textit{et al.} \cite{SJTCE},
growth and morphology of nanocrystals, see Woehl \textit{et al.} \cite{WPEA} and Kaganer \textit{et al.\/} \cite{KPS},
epitaxial $\text{Si}_{1-\alpha}\text{Ge}_{\alpha}/\text{Si}(001)$ islands, see Budiman and Ruda~\cite{BR},
growth of graphene on Ir(111), see Coraux \textit{et al.\/} \cite{CNDEBWBMGPM}; and gold nanoparticles on a silicon substrate, see Winkler \textit{et al.} \cite{WWLGF};
blood clotting, see Guy \textit{et al.\/} \cite{GFK} and Rouleau formation, see Samsel and Perelson~\cite{SPI,SPII} and polymer growth of proteins in biopharmaceuticals,
see Galina \textit{et al.\/} \cite{GLK} or Zidar \textit{et al.\/} \cite{ZKR}.

Our paper is structured as follows. In Section~\ref{sec:trees} we introduce the algebra of planar binary trees.
We solve the abstract equation~\eqref{eq:quadnonassoc} in this algebra and present three different equivalent solution forms.
We apply our abstract results from Section~\ref{sec:trees} to coagulation examples in Section~\ref{sec:coagulation}. 
In Section~\ref{sec:simulation} we demonstrate that the solution flow can be represented as a linearised flow, and
we outline how the binary tree expansion solution can be approximated and used for numerical simulation.
Finally, in Section~\ref{sec:discussion} we discuss the many future directions to which our results herein extend.

\section{Non-associative algebras and binary trees}\label{sec:trees}                                                                           
We have seen that the solution~\eqref{eq:formalsoln} to the abstract evolutionary quadratic equation~\eqref{eq:quadnonassoc},
representing Smoluchowski's coagulation equation, is formally given as a series in non-associative products `$\star$' of the data $\xi$.
The coefficients of the series have a simple time-dependence.
We are thus lead to consider the non-commutative, non-associative algebra with a single generator, and product `$\star$'. 
Such algebras have a representation in the real algebra of planar binary trees. Hence our primary focus herein is on expansions in such algebras.

Let $\Hb=\Hb(\xi,\star)$ denote the real non-commutative, non-associative algebra generated by the function $\xi$, with the product `$\star$'.
We do not require $\Hb$ to be unital. The elements of $\Hb$ are polynomials or series of the form,
\begin{equation}\label{eq:seriesindeterminants}
  \gf=\hat\gf_1\,\xi+\hat\gf_{11}\,(\xi\star\xi)+\hat\gf_{1(11)}\,\bigl(\xi\star(\xi\star\xi)\bigr)+\cdots, 
\end{equation}
where all the coefficients $\hat\gf_1$, $\hat\gf_{11}$, $\hat\gf_{1(11)}$ and so forth, are real. Our focus in this section is on the combinatorial structure of such series.
In Section~\ref{sec:coagulation} we outline in what sense such series converge in $\Hb$, corresponding to a restriction on the coefficients $\hat\gf_1$, $\hat\gf_{11}$, etc. 
An immediate insight from the form of~\eqref{eq:seriesindeterminants} is that the elements of $\Hb$ are characterised by the
binary parenthesisation of strings, and linear combinations thereof. The binary parenthesisation of strings is isomorphically mapped to planar binary trees;
see Stanley~\cite[Ch.~6]{Stanley}. See Tables~\ref{table:treesupto3} and \ref{table:treesof4} for examples of such trees.
As a consequence $\Hb$ is isomorphically mapped to the real algebra of binary trees $\R\lb\mathbb T\rb$, which we outline in detail presently.
The non-commutative, non-associative product `$\star$' on $\Hb$ is replaced by the grafting of pairs of trees at their roots---see Procedure~\ref{proc:grafting} just below.
Indeed, as already highlighted in the introduction, we can represent series such as~\eqref{eq:seriesindeterminants} in the form, 
\begin{equation*}
  \gf=\hat\gf_1\,\xi({\tiny \begin{forest} for tree={grow'=90, parent anchor=center,child anchor=center} [] \path[fill=black] (.anchor) circle[radius=1.5pt]; \end{forest}})
  +\hat\gf_{11}\,\xi(\raisebox{-1pt}{{\tiny \begin{forest} for tree={grow'=90, parent anchor=center,child anchor=center, l=0cm,inner ysep=0cm,edge+=thick} [[][]] \end{forest}}})
  +\hat\gf_{1(11)}\,\xi\Bigl(\raisebox{-4pt}{{\tiny \begin{forest} for tree={grow'=90, parent anchor=center,child anchor=center, l=0cm,inner ysep=0cm,edge+=thick} [[][[][]]] \end{forest}}}\Bigr)+\cdots.
\end{equation*}
We can thus regard $\xi$ as a homomorphic map from $\R\lb\mathbb T\rb$ to $\Hb$ such that for any pair of trees $\tau_1$ and $\tau_2$, we have,
\begin{equation*}
  \xi\Bigl(\raisebox{-5pt}{{\scriptsize \begin{forest} for tree={grow'=90, parent anchor=center,child anchor=south, l=0cm,inner ysep=2pt,edge+=thick, s sep=0mm} [[$\tau_1$][$\tau_2$]] \end{forest}}}\Bigr)
  =\xi(\tau_1)\star\xi(\tau_2). 
\end{equation*}
In the abstract context here, we are seeking solutions of the evolutionary quadratic equation~\eqref{eq:quadnonassoc} in $\Hb$.
However with this homomorphic property of $\xi$ in mind, if we pullback equation~\eqref{eq:quadnonassoc} from $\Hb$ to $\R\lb\mathbb T\rb$,
the result is the following Smoluchowski tree equation on $\R\la\mathbb T\ra$.
\begin{definition}[Smoluchowski tree equation]\label{def:Smoluchowskitree}
  We define the \emph{Smoluchowski tree equation} as the evolutionary quadratic equation for $\gf=\gf(t)$ in $\R\la\mathbb T\ra$ given by,
  \begin{equation}\label{eq:coagtrees}
    \pa_t\gf=\raisebox{-5pt}{{\scriptsize \begin{forest} for tree={grow'=90, parent anchor=center,child anchor=south, l=0cm,inner ysep=2pt,edge+=thick, s sep=1mm} [[$\gf$][$\gf$]] \end{forest}}}.
  \end{equation}
  We augment~\eqref{eq:coagtrees} with the data
  $\gf(0)={\tiny \begin{forest} for tree={grow'=90, parent anchor=center,child anchor=center} [] \path[fill=black] (.anchor) circle[radius=1.5pt]; \end{forest}}$.
\end{definition}
\begin{remark}
  With a slight abuse of notation we use the same label `$\gf$' to represent the solution to the Smoluchowksi equation in its different guises.
  It is the solution to~\eqref{eq:quadnonassoc} when $\gf=\gf(s;t)$ is a function, which we consider in more detail in Section~\ref{sec:coagulation},
  or the solution $\gf\in\R\la\mathbb T\ra$ to the Smoluchowski tree equation~\eqref{eq:coagtrees}.
  It will always be clear from the context, which solution form we refer to.
\end{remark}
Our goal herein is to find a solution to initial value problem for the Smoluchowski tree equation~\eqref{eq:coagtrees} in Definition~\ref{def:Smoluchowskitree}.
For now, we focus on the structure of the real algebra of planar binary trees $\R\lb\mathbb T\rb$. Here $\mathbb T$ denotes the
set/forest of all rooted planar binary trees. The root is the bottom vertex and the trees grow upwards from there.
\begin{definition}[Grade]
The \emph{grade} $|\tau|$ of any binary tree $\tau\in\mathbb T$, is the number of vertices it possesses.
\end{definition}
\begin{remark}[Vertices]
  We nominate the number of vertices of any tree as the number of internal branchings or
  $\raisebox{-1pt}{{\tiny \begin{forest} for tree={grow'=90, parent anchor=center,child anchor=center, l=0cm,inner ysep=0cm,edge+=thick} [[][]] \end{forest}}}$'s
  the tree possesses. Other authors include the number of free branches in the vertex count as well, depending on the context.
  See Tables~\ref{table:treesupto3} and \ref{table:treesof4} for further clarification.
\end{remark}
Two natural procedures for generating rooted planar binary trees are as follows.
\begin{procedure}[Grafting]\label{proc:grafting}
  We can generate all the trees of grade $n$ from the trees of grade $0$ through $n-1$ by root-grafting as follows.
  We graft at the root all trees $\tau_1$ and $\tau_2$ such that $|\tau_1|+|\tau_2|=n-1$ with $|\tau_1|$ cycling through $0$, $1$, $\ldots$, $n-1$.
  This procedure exhaustively generates all the possible trees of grade $n$. 
  For example, we can generate all the binary trees of grade $4$ in Table~\ref{table:treesof4}, using all the trees in Table~\ref{table:treesupto3}, by:
  grafting the grade $0$ tree 
  `{\tiny \begin{forest} for tree={grow'=90, parent anchor=center,child anchor=center} [] \path[fill=black] (.anchor) circle[radius=1.5pt]; \end{forest}}'
  on the left with all the grade $3$ trees on the right; grafting the grade $1$ tree 
  `\raisebox{-1pt}{{\tiny \begin{forest} for tree={grow'=90, parent anchor=center,child anchor=center, l=0cm,inner ysep=0cm,edge+=thick} [[][]] \end{forest}}}'
  on the left with both the grade $2$ trees on the right; grafting both grade $2$ trees on the left with the grade $1$ tree 
  `\raisebox{-1pt}{{\tiny \begin{forest} for tree={grow'=90, parent anchor=center,child anchor=center, l=0cm,inner ysep=0cm,edge+=thick} [[][]] \end{forest}}}'
  on the right; and then finally grafting all the grade $3$ trees on the left with the grade $0$ tree
  `{\tiny \begin{forest} for tree={grow'=90, parent anchor=center,child anchor=center} [] \path[fill=black] (.anchor) circle[radius=1.5pt]; \end{forest}}'
  on the right.
\end{procedure}
\begin{remark}[Grafting and concatentation]
  As indicated above, one natural product on the real algebra of rooted planar trees $\R\la\mathbb T\ra$
  is grafting in the manner described in Procedure~\ref{proc:grafting}; see for example Lundervold and Munthe--Kaas~\cite{LM-Kbackward}.
  We endow $\R\la\mathbb T\ra$ with this product. It is akin to a concatenation product of parenthesised strings. 
\end{remark}
\begin{remark}[General grafting]
  In principle we can graft a planar binary tree $\tau^\prime$ to any free branch or even vertex of another planar binary tree. 
  However, we do not require this procedure herein.
\end{remark}
\begin{remark}[Catalan numbers]
  There is an intimate link between Catalan numbers and planar binary trees. Indeed, Stanley~\cite[Ex.~6.19]{Stanley} gives $66$ characterisations
  of Catalan numbers. For $n\in\mathbb N_0$, the $n$th Catalan number is,
  \begin{equation*}
     C_n\coloneqq\frac{1}{n+1}\begin{pmatrix} 2n \\ n\end{pmatrix}.
  \end{equation*}
  The Catalan number $C_n$ counts the number of planar binary trees with $n$ vertices, in other words, the number of binary trees generated
  at each grade by the grafting procedure above. Planar binary trees, or variants thereof, are often nominated as Catalan trees.
\end{remark}
\begin{procedure}[Branching]\label{proc:branching}
  We can generate all the trees of grade $n$ from the trees of grade $n-1$ as follows. To each tree $\tau$ at grade $n-1$ we attach a single branch
  `\raisebox{-1pt}{{\tiny \begin{forest} for tree={grow'=90, parent anchor=center,child anchor=center, l=0cm,inner ysep=0cm,edge+=thick} [[][]] \end{forest}}}',
  successively, to each free end of $\tau$. This procedure also exhaustively generates all the possible trees of grade $n$. However some trees at grade $n$ are
  multiply generated by those at grade $n-1$ by this procedure. The multiplicity of any tree generated by this procedure corresponds to the weight character
  of that tree; see Definition~\ref{def:weightchar} and Lemmas~\ref{lemma:grafting} and \ref{lemma:branchpower} below.
  For example, in Table~\ref{table:treesupto3}, if we successively add a branch
  `\raisebox{-1pt}{{\tiny \begin{forest} for tree={grow'=90, parent anchor=center,child anchor=center, l=0cm,inner ysep=0cm,edge+=thick} [[][]] \end{forest}}}'
  to each free end of the trees of grade $2$, we generate the five binary trees of grade $3$ shown, however the middle symmetric tree is generated
  with multiplicity~$2$.
\end{procedure}
\begin{remark}[Composition and budding]
The procedure for branching outlined is a simple version of the composition of trees; see eg.\/ Lundervold and Munthe--Kaas~\cite{LM-Kbackward,LM-K}.
More general trees and linear combinations of trees than just
`\raisebox{-1pt}{{\tiny \begin{forest} for tree={grow'=90, parent anchor=center,child anchor=center, l=0cm,inner ysep=0cm,edge+=thick} [[][]] \end{forest}}}'
can be attached to free branches in the manner outlined in Procedure~\ref{proc:branching}.
Branching is also related to ``budding''. See Definition~\ref{def:budding} in Section~\ref{sec:simulation} and items (iii) and (v) in Section~\ref{sec:discussion}.
\end{remark}
\begin{remark}[Empty tree]
  Note, the empty tree $\emptyset$ doesn't play a role in either of the tree generating procedures above.
  Indeed, the empty tree $\emptyset$ plays the role of a unit in $\R\la\mathbb T\ra$, though we do not explicitly use this here.
\end{remark}
We define the \emph{weight character} associated with any given tree in $\R\lb\mathbb T\rb$ as follows.
\begin{definition}[Weight character]\label{def:weightchar}
  We define the \emph{weight character} $\omega\colon\R\lb\mathbb T\rb\to\Qb$ of any tree $\tau\in\R\lb\mathbb T\rb$ recursively as follows
  (by convention $\omega({\tiny \begin{forest} for tree={grow'=90, parent anchor=center,child anchor=center} [] \path[fill=black] (.anchor) circle[radius=1.5pt]; \end{forest}})\coloneqq1$): 
  \begin{align*}
    \omega(\raisebox{-1pt}{{\tiny \begin{forest} for tree={grow'=90, parent anchor=center,child anchor=center, l=0cm,inner ysep=0cm,edge+=thick} [[][]] \end{forest}}})
    &=\begin{pmatrix} 0\\ 0\end{pmatrix}
    \intertext{and then}
    \omega\Bigl(\raisebox{-5pt}{{\scriptsize \begin{forest} for tree={grow'=90, parent anchor=center,child anchor=south, l=0cm,inner ysep=2pt,edge+=thick, s sep=0mm} [[$\tau_1$][$\tau_2$]] \end{forest}}}\Bigr)
    &=\begin{pmatrix} |\tau_1|+|\tau_2|\\ |\tau_1|\end{pmatrix} \omega(\tau_1)\,\omega(\tau_2).
  \end{align*}
\end{definition}
Thus in essence, to compute the weight character of a given binary tree, we split the tree at its root and compute the product of
the weight characters of the two subtrees $\tau_1$ and $\tau_2$ generated and include the Leibniz factor $|\tau_1|+|\tau_2|$ choose $|\tau_1|$.
The weight characters of $\tau_1$ and $\tau_2$ are computed by propagating this process.
Example weight characters for trees up to grade $4$ are given in Tables~\ref{table:treesupto3} and \ref{table:treesof4}.
Therein we abbreviate products of Leibniz coefficients by,
\begin{equation*}
\begin{pmatrix} m_1\\k_1\end{pmatrix}\begin{pmatrix} m_2\\k_2\end{pmatrix}\cdots \begin{pmatrix} m_n\\k_n \end{pmatrix} = \begin{pmatrix} m_1 & m_2 & \cdots & m_n\\ k_1 & k_2 & \cdots & k_n\end{pmatrix}.
\end{equation*}
The weight character naturally arises in the process of tree generation by branching, as described above, and outlined below. See Lemmas~\ref{lemma:grafting} and \ref{lemma:branchpower}. 
\begin{definition}[Exponential series]\label{def:expseries}
We define an \emph{exponential tree series} in $\R\lb\mathbb T\rb$ as an expansion of the form,
\begin{equation}\label{eq:exptreeseries}
\ff=\sum_{\tau\in\mathbb T\backslash\emptyset}\frac{\omega(\tau)}{|\tau|!}\,\hat{\ff}_\tau\cdot\tau,
\end{equation}
where the $\hat{\ff}_\tau$ are the real coefficients associated with each tree $\tau\in\mathbb T\backslash\emptyset$, and
$\omega(\tau)$ and $|\tau|$ are respectively the weight character and grade of $\tau$.
\end{definition}
Two further operators are useful to our development.

\begin{definition}[Branching operator]\label{def:branchingoperator}
  We define the linear \emph{branching operator} $B_{\vee}\colon\R\lb\mathbb T\rb\to\R\lb\mathbb T\rb$ as the operator that acts on any
  tree $\tau\in\mathbb T$ by successively, additively attaching a branch
  `\raisebox{-1pt}{{\tiny \begin{forest} for tree={grow'=90, parent anchor=center,child anchor=center, l=0cm,inner ysep=0cm,edge+=thick} [[][]] \end{forest}}}'
  to each free end of the tree $\tau$, thus generating a sum of the corresponding trees at the next grade.
\end{definition}

\begin{remark}\label{rmk:extend}
  We can extend the branching operator to $B_{\tau^\prime}$, for any $\tau^\prime\in\mathbb T$, so that it becomes the
  operator that successively, additively attaches the tree $\tau^\prime$ to each free end of the tree $\tau$. We can also
  linearly extend $B_{\tau^\prime}$ so that $B_{a\cdot\tau_1^\prime+b\cdot\tau_2^\prime}=a\cdot B_{\tau_1^\prime}+b\cdot B_{\tau_2^\prime}$ for any $a,b\in\R$.  
\end{remark}

\begin{example}\label{ex:branching}
  Some examples of the action of the branching operator are as follows, we observe,
  \begin{align*}
    B_{\vee}({\tiny \begin{forest} for tree={grow'=90, parent anchor=center,child anchor=center} [] \path[fill=black] (.anchor) circle[radius=1.5pt]; \end{forest}})
    =&\;\raisebox{-1pt}{{\tiny \begin{forest} for tree={grow'=90, parent anchor=center,child anchor=center, l=0cm,inner ysep=0cm,edge+=thick} [[][]] \end{forest}}},\\    
    B_{\vee}(\raisebox{-1pt}{{\tiny \begin{forest} for tree={grow'=90, parent anchor=center,child anchor=center, l=0cm,inner ysep=0cm,edge+=thick} [[][]] \end{forest}}})
    =&\;\raisebox{-4pt}{{\tiny \begin{forest} for tree={grow'=90, parent anchor=center,child anchor=center, l=0cm,inner ysep=0cm,edge+=thick} [[][[][]]] \end{forest}}}
    +\raisebox{-4pt}{{\tiny \begin{forest} for tree={grow'=90, parent anchor=center,child anchor=center, l=0cm,inner ysep=0cm,edge+=thick} [[[][]][]] \end{forest}}},\\
    B_{\vee}^2(\raisebox{-1pt}{{\tiny \begin{forest} for tree={grow'=90, parent anchor=center,child anchor=center, l=0cm,inner ysep=0cm,edge+=thick} [[][]] \end{forest}}})
    =&\;B_\vee\Bigl(\raisebox{-4pt}{{\tiny \begin{forest} for tree={grow'=90, parent anchor=center,child anchor=center, l=0cm,inner ysep=0cm,edge+=thick} [[][[][]]] \end{forest}}}
    +\raisebox{-4pt}{{\tiny \begin{forest} for tree={grow'=90, parent anchor=center,child anchor=center, l=0cm,inner ysep=0cm,edge+=thick} [[[][]][]] \end{forest}}}\Bigr)\\
    =&\;\raisebox{-8pt}{{\tiny \begin{forest} for tree={grow'=90, parent anchor=center,child anchor=center, l=0cm,inner ysep=0cm,edge+=thick} [[][[][[][]]]] \end{forest}}}
    +\raisebox{-8pt}{{\tiny \begin{forest} for tree={grow'=90, parent anchor=center,child anchor=center, l=0cm,inner ysep=0cm,edge+=thick} [[][[[][]][]]] \end{forest}}}
    +2\cdot\raisebox{-8pt}{{\tiny \begin{forest} for tree={grow'=90, parent anchor=center,child anchor=center, l=0cm,inner ysep=0cm,edge+=thick} [[[][]][[][]]] \end{forest}}}
    +\raisebox{-8pt}{{\tiny \begin{forest} for tree={grow'=90, parent anchor=center,child anchor=center, l=0cm,inner ysep=0cm,edge+=thick} [[[][[][]]][]] \end{forest}}}
    +\raisebox{-8pt}{{\tiny \begin{forest} for tree={grow'=90, parent anchor=center,child anchor=center, l=0cm,inner ysep=0cm,edge+=thick} [[[[][]][]][]] \end{forest}}}.
  \end{align*}
\end{example}

\begin{table}
  \caption{We list all the rooted planar binary trees up to grade $3$. For the grades in the left column, we list each tree $\tau\in\mathbb T$ in the second column,
    its corresponding weight $\omega(\tau)$ in the third column, the number of symmetries $\sigma(\tau)$ it possesses (see Remark~\ref{rmk:nonplanartrees}) in the fourth column
    and its word-coding in terms of levels in the final column.}
\label{table:treesupto3}
\begin{center}
   \setlength{\baselineskip}{2\baselineskip}
\begin{tabular}{|c|c|c|c|c|}
\hline
$\phantom{\biggl|}$ grade $\phantom{\biggl|}$ & tree & weight & symm.\/ & levels \\ 
\hline
0 & {\tiny \begin{forest} for tree={grow'=90, parent anchor=center,child anchor=center} [] \path[fill=black] (.anchor) circle[radius=1.5pt]; \end{forest}} & 1 & 0 & 0  \\[5pt]
1 & {\tiny \begin{forest} for tree={grow'=90, parent anchor=center,child anchor=center, l=0cm,inner ysep=0cm,edge+=thick} [[][]] \end{forest}} &
\raisebox{\depth}{{\scriptsize $\begin{pmatrix} 0\\ 0\end{pmatrix}$}}& 0 & 1 \\\hline
&&&&\\[-5pt]
2 & {\tiny \begin{forest} for tree={grow'=90, parent anchor=center,child anchor=center, l=0cm,inner ysep=0cm,edge+=thick} [[][[][]]] \end{forest}} &
\raisebox{\depth}{{\scriptsize $\begin{pmatrix} 1\!\!&\!\! 0\\ 0 \!\!&\!\! 0\end{pmatrix}$}}& 1 & 12\\
&&&&\\[-5pt]
2 & {\tiny \begin{forest} for tree={grow'=90, parent anchor=center,child anchor=center, l=0cm,inner ysep=0cm,edge+=thick} [[[][]][]] \end{forest}} &
\raisebox{\depth}{{\scriptsize $\begin{pmatrix} 1\!\!&\!\! 0\\ 1 \!\!&\!\! 0\end{pmatrix}$}}& 1 & 21 \\\hline
&&&&\\[-5pt]
3 & {\tiny \begin{forest} for tree={grow'=90, parent anchor=center,child anchor=center, l=0cm,inner ysep=0cm,edge+=thick} [[][[][[][]]]] \end{forest}} &
\raisebox{\depth}{{\scriptsize $\begin{pmatrix} 2\!\!&\!\! 1\!\! &\!\! 0 \\ 0 \!\!&\!\! 0 \!\!&\!\! 0\end{pmatrix}$}}& 2 & 123\\
  3 & {\tiny \begin{forest} for tree={grow'=90, parent anchor=center,child anchor=center, l=0cm,inner ysep=0cm,edge+=thick} [[][[[][]][]]] \end{forest}} &
\raisebox{\depth}{{\scriptsize $\begin{pmatrix} 2\!\!&\!\! 1\!\! &\!\! 0 \\ 0 \!\!&\!\! 1 \!\!&\!\! 0\end{pmatrix}$}}& 2 & 132\\[7pt]
3 & {\tiny \begin{forest} for tree={grow'=90, parent anchor=center,child anchor=center, l=0cm,inner ysep=0cm,edge+=thick} [[[][]][[][]]] \end{forest}} &
\raisebox{\depth}{{\scriptsize $\begin{pmatrix} 2\!\!&\!\! 0\!\! &\!\! 0 \\ 1 \!\!&\!\! 0 \!\!&\!\! 0\end{pmatrix}$}}& 0 & 212\\
3 & {\tiny \begin{forest} for tree={grow'=90, parent anchor=center,child anchor=center, l=0cm,inner ysep=0cm,edge+=thick} [[[][[][]]][]] \end{forest}} &
\raisebox{\depth}{{\scriptsize $\begin{pmatrix} 2\!\!&\!\! 1\!\! &\!\! 0 \\ 2 \!\!&\!\! 0 \!\!&\!\! 0\end{pmatrix}$}}& 2 & 231\\
3 & {\tiny \begin{forest} for tree={grow'=90, parent anchor=center,child anchor=center, l=0cm,inner ysep=0cm,edge+=thick} [[[[][]][]][]] \end{forest}} &
\raisebox{\depth}{{\scriptsize $\begin{pmatrix} 2\!\!&\!\! 1\!\! &\!\! 0 \\ 2 \!\!&\!\! 1 \!\!&\!\! 0\end{pmatrix}$}}& 2 & 321\\\hline
\end{tabular}
\end{center}
\end{table}

\begin{definition}[Grade operator]\label{def:gradingoperator}
  The \emph{grade operator} $G$ on $\R\lb\mathbb T\rb$ is given for any $\tau\in\mathbb T$ by,
  \begin{equation*}
     G\colon\tau\mapsto\frac{1}{|\tau|}\cdot\tau.
  \end{equation*}
  Naturally, $G$ is an endomorphism on $\R\lb\mathbb T\rb$.  
\end{definition}

\begin{example}
The grading operator normalises the given tree by its grade, so that for example,
\begin{equation*}
  G\Bigl(\raisebox{-4pt}{{\tiny \begin{forest} for tree={grow'=90, parent anchor=center,child anchor=center, l=0cm,inner ysep=0cm,edge+=thick} [[[][]][]] \end{forest}}}\Bigr)
  =\frac12\cdot\raisebox{-4pt}{{\tiny \begin{forest} for tree={grow'=90, parent anchor=center,child anchor=center, l=0cm,inner ysep=0cm,edge+=thick} [[[][]][]] \end{forest}}}
  \quad\text{and}\quad
  G\Biggl(\raisebox{-8pt}{{\tiny \begin{forest} for tree={grow'=90, parent anchor=center,child anchor=center, l=0cm,inner ysep=0cm,edge+=thick} [[[][[][]]][]] \end{forest}}}\Biggr)
=\frac13\cdot\raisebox{-8pt}{{\tiny \begin{forest} for tree={grow'=90, parent anchor=center,child anchor=center, l=0cm,inner ysep=0cm,edge+=thick} [[[][[][]]][]] \end{forest}}}.
\end{equation*}
\end{example}

\begin{example}\label{ex:adapted}
It is also natural to combine the grading and branching operators, and so we set, $\hat{B}_{\vee}\coloneqq G\circ B_{\vee}$. 
Thus, adapting Example~\ref{ex:branching}, we observe,
  \begin{align*}
    \hat{B}_{\vee}^2(\raisebox{-1pt}{{\tiny \begin{forest} for tree={grow'=90, parent anchor=center,child anchor=center, l=0cm,inner ysep=0cm,edge+=thick} [[][]] \end{forest}}})
    =&\;\tfrac12 B_\vee\Bigl(\raisebox{-4pt}{{\tiny \begin{forest} for tree={grow'=90, parent anchor=center,child anchor=center, l=0cm,inner ysep=0cm,edge+=thick} [[][[][]]] \end{forest}}}
    +\raisebox{-4pt}{{\tiny \begin{forest} for tree={grow'=90, parent anchor=center,child anchor=center, l=0cm,inner ysep=0cm,edge+=thick} [[[][]][]] \end{forest}}}\Bigr)\\
    =&\;\tfrac16\Biggl(\!\!\raisebox{-8pt}{{\tiny \begin{forest} for tree={grow'=90, parent anchor=center,child anchor=center, l=0cm,inner ysep=0cm,edge+=thick} [[][[][[][]]]] \end{forest}}}
    +\raisebox{-8pt}{{\tiny \begin{forest} for tree={grow'=90, parent anchor=center,child anchor=center, l=0cm,inner ysep=0cm,edge+=thick} [[][[[][]][]]] \end{forest}}}
    +2\cdot\raisebox{-8pt}{{\tiny \begin{forest} for tree={grow'=90, parent anchor=center,child anchor=center, l=0cm,inner ysep=0cm,edge+=thick} [[[][]][[][]]] \end{forest}}}
    +\raisebox{-8pt}{{\tiny \begin{forest} for tree={grow'=90, parent anchor=center,child anchor=center, l=0cm,inner ysep=0cm,edge+=thick} [[[][[][]]][]] \end{forest}}}
    +\raisebox{-8pt}{{\tiny \begin{forest} for tree={grow'=90, parent anchor=center,child anchor=center, l=0cm,inner ysep=0cm,edge+=thick} [[[[][]][]][]] \end{forest}}}\!\!\Biggr).
  \end{align*}
\end{example}
Naturally, since the non-associative product `$\star$' is bilinear, so is the grafting operator.
In other words, for example, for any three trees $\tau_1,\tau_2,\tau_3\in\mathbb T$, we have,
\begin{equation}\label{eq:bilinear}
 \raisebox{-5pt}{{\scriptsize \begin{forest} for tree={grow'=90, parent anchor=center,child anchor=south, l=0cm,inner ysep=2pt,edge+=thick, s sep=0mm} [[$\tau_1+\tau_2$][$\tau_3$]] \end{forest}}}
 =\raisebox{-5pt}{{\scriptsize \begin{forest} for tree={grow'=90, parent anchor=center,child anchor=south, l=0cm,inner ysep=2pt,edge+=thick, s sep=0mm} [[$\tau_1$][$\tau_3$]] \end{forest}}}
 +\raisebox{-5pt}{{\scriptsize \begin{forest} for tree={grow'=90, parent anchor=center,child anchor=south, l=0cm,inner ysep=2pt,edge+=thick, s sep=0mm} [[$\tau_2$][$\tau_3$]] \end{forest}}}.
\end{equation}
The branching operator also acts like a derivation in the following sense. For any two trees $\tau_1,\tau_2\in\mathbb T$, we observe,
\begin{equation}\label{eq:derivation}
 B_{\vee}\Bigl(\raisebox{-5pt}{{\scriptsize \begin{forest} for tree={grow'=90, parent anchor=center,child anchor=south, l=0cm,inner ysep=2pt,edge+=thick, s sep=0mm} [[$\tau_1$][$\tau_2$]] \end{forest}}}\Bigr)
 =\raisebox{-5pt}{{\scriptsize \begin{forest} for tree={grow'=90, parent anchor=center,child anchor=south, l=0cm,inner ysep=2pt,edge+=thick, s sep=0mm} [[$B_\vee(\tau_1)$][$\tau_2$]] \end{forest}}}
 +\raisebox{-5pt}{{\scriptsize \begin{forest} for tree={grow'=90, parent anchor=center,child anchor=south, l=0cm,inner ysep=2pt,edge+=thick, s sep=0mm} [[$\tau_1$][$B_\vee(\tau_2)$]] \end{forest}}}.
\end{equation}
A crucial component of our main result herein is the following identity. For convenience, for all $n\in\mathbb N_0$,
let $\gf_n$ denote the weighted sum of all trees of grade $n$, so,
\begin{equation}\label{eq:gndef}
   \gf_n\coloneqq\sum_{|\tau|=n} \omega(\tau)\cdot\tau.
\end{equation}
\begin{lemma}[Grafting identity]\label{lemma:grafting}
  The following \emph{grafting identity} holds for all $n\in\mathbb N_0$:
  \begin{equation*}
    \gf_{n+1}=\sum_{k=0}^n\begin{pmatrix} n\\ k\end{pmatrix}
    \cdot\raisebox{-5pt}{{\scriptsize \begin{forest} for tree={grow'=90, parent anchor=center,child anchor=south, l=0cm,inner ysep=2pt,edge+=thick, s sep=0mm} [[$\gf_{n-k}$][$\gf_k$]] \end{forest}}}.
  \end{equation*}
\end{lemma}
\begin{proof}
  Using that the grafting operator is bilinear and the properties of the weight character given in Definition~\ref{def:weightchar}, we observe the right-hand side equals,
  \begin{align*}
    \sum_{k=0}^n&\sum_{|\tau_1|=n-k}\sum_{|\tau_2|=k} \begin{pmatrix} |\tau_1|+|\tau_2|\\ |\tau_1|\end{pmatrix}\omega(\tau_1)\omega(\tau_2)\cdot
    \raisebox{-5pt}{{\scriptsize \begin{forest} for tree={grow'=90, parent anchor=center,child anchor=south, l=0cm,inner ysep=2pt,edge+=thick, s sep=0mm} [[$\tau_1$][$\tau_2$]] \end{forest}}}\\
    =&\;\sum_{k=0}^n\sum_{|\tau_1|=n-k}\sum_{|\tau_2|=k}
    \omega\Bigl(\raisebox{-5pt}{{\scriptsize \begin{forest} for tree={grow'=90, parent anchor=center,child anchor=south, l=0cm,inner ysep=2pt,edge+=thick, s sep=0mm} [[$\tau_1$][$\tau_2$]] \end{forest}}}\Bigr)\cdot
    \raisebox{-5pt}{{\scriptsize \begin{forest} for tree={grow'=90, parent anchor=center,child anchor=south, l=0cm,inner ysep=2pt,edge+=thick, s sep=0mm} [[$\tau_1$][$\tau_2$]] \end{forest}}}\\
    =&\;\sum_{|\tau|=n+1} \omega(\tau)\cdot\tau,
  \end{align*}
giving the result.\qed
\end{proof}
As a consequence, we have the following expansion for any power of the branching operator $B_{\vee}$
acting on the tree `{\tiny \begin{forest} for tree={grow'=90, parent anchor=center,child anchor=center} [] \path[fill=black] (.anchor) circle[radius=1.5pt]; \end{forest}}'.
\begin{lemma}\label{lemma:branchpower}
  For any $n\in\mathbb N_0$, we have,
  \begin{equation*}
    B_{\vee}^n({\tiny \begin{forest} for tree={grow'=90, parent anchor=center,child anchor=center} [] \path[fill=black] (.anchor) circle[radius=1.5pt]; \end{forest}})=\sum_{|\tau|=n}\omega(\tau)\cdot\tau.
  \end{equation*}
  or equivalently, $B_{\vee}^n({\tiny \begin{forest} for tree={grow'=90, parent anchor=center,child anchor=center} [] \path[fill=black] (.anchor) circle[radius=1.5pt]; \end{forest}})=\gf_n$.
\end{lemma}
\begin{proof}
  We prove the result by induction. Assume that the result holds for all powers $B_\vee^\ell$ with $\ell=0,1,\ldots,n$.
  Note that for any of the values $\ell=0,1,\ldots,n-1$, the statement above is equivalent to the statement,
  \begin{equation}\label{eq:inductionassumption}
    B_\vee\Biggl(\sum_{|\tau|=\ell}\omega(\tau)\cdot\tau\Biggr)=\sum_{|\tau|=\ell+1}\omega(\tau)\cdot\tau,
  \end{equation}
  or in other words, $B_\vee(\gf_\ell)=\gf_{\ell+1}$. First, we observe that,
  \begin{equation*}
    B_\vee^{n+1}({\tiny \begin{forest} for tree={grow'=90, parent anchor=center,child anchor=center} [] \path[fill=black] (.anchor) circle[radius=1.5pt]; \end{forest}})
    =B_\vee\bigl(B_\vee^n({\tiny \begin{forest} for tree={grow'=90, parent anchor=center,child anchor=center} [] \path[fill=black] (.anchor) circle[radius=1.5pt]; \end{forest}})\bigr)
    =B_\vee\Biggl(\sum_{|\tau|=n}\omega(\tau)\cdot\tau\Biggr)
    =B_\vee(\gf_n),
  \end{equation*}
  using our induction assumption.
  Second, by direct computation, using Lemma~\ref{lemma:grafting}, we observe that,
  \begin{equation*}
    B_\vee(\gf_n)=B_\vee\Biggl(\sum_{k=0}^{n-1}\begin{pmatrix} n-1\\ k\end{pmatrix}
    \cdot\raisebox{-5pt}{{\scriptsize \begin{forest} for tree={grow'=90, parent anchor=center,child anchor=south, l=0cm,inner ysep=2pt,edge+=thick, s sep=0mm} [[$\gf_{n-1-k}$][$\gf_k$]] \end{forest}}}\Biggr).
  \end{equation*}
  Third, using the grafting bilinear property \eqref{eq:bilinear} and the derivation property of $B_\vee$ from \eqref{eq:derivation}, we have,
  \begin{equation*}
     B_\vee(\gf_n)=\sum_{k=0}^{n-1}\begin{pmatrix} n-1\\ k\end{pmatrix}\cdot\biggl(
       \raisebox{-5pt}{{\scriptsize \begin{forest} for tree={grow'=90, parent anchor=center,child anchor=south, l=0cm,inner ysep=2pt,edge+=thick, s sep=0mm} [[$B_\vee(\gf_{n-1-k})$][$\gf_k$]] \end{forest}}}
       +\raisebox{-5pt}{{\scriptsize \begin{forest} for tree={grow'=90, parent anchor=center,child anchor=south, l=0cm,inner ysep=2pt,edge+=thick, s sep=0mm} [[$\gf_{n-1-k}$][$B_\vee(\gf_k)$]] \end{forest}}}\biggr).
  \end{equation*}
  Fourth, substituting for $\gf_{n-1-k}$ and $\gf_k$ and using the induction assumption~\eqref{eq:inductionassumption}, we observe $B_\vee(\gf_n)$ equals,
  \begin{align*}
  &\;\sum_{k=0}^{n-1}\begin{pmatrix} n-1\\ k\end{pmatrix}\sum_{|\tau_1|=n-k}\sum_{|\tau_2|=k}\omega(\tau_1)\omega(\tau_2)
   \cdot\raisebox{-5pt}{{\scriptsize \begin{forest} for tree={grow'=90, parent anchor=center,child anchor=south, l=0cm,inner ysep=2pt,edge+=thick, s sep=0mm} [[$\tau_1$][$\tau_2$]] \end{forest}}}\\
   &\;+\sum_{k=0}^{n-1}\begin{pmatrix} n-1\\ k\end{pmatrix}\sum_{|\tau_1|=n-1-k}\sum_{|\tau_2|=k+1}\omega(\tau_1)\omega(\tau_2)
     \cdot\raisebox{-5pt}{{\scriptsize \begin{forest} for tree={grow'=90, parent anchor=center,child anchor=south, l=0cm,inner ysep=2pt,edge+=thick, s sep=0mm} [[$\tau_1$][$\tau_2$]] \end{forest}}}\\
   =&\;\sum_{k=0}^{n-1}\begin{pmatrix} n-1\\ k\end{pmatrix}\sum_{|\tau_1|=n-k}\sum_{|\tau_2|=k}\omega(\tau_1)\omega(\tau_2)
   \cdot\raisebox{-5pt}{{\scriptsize \begin{forest} for tree={grow'=90, parent anchor=center,child anchor=south, l=0cm,inner ysep=2pt,edge+=thick, s sep=0mm} [[$\tau_1$][$\tau_2$]] \end{forest}}}\\
     &\;+\sum_{k=1}^{n}\begin{pmatrix} n-1\\ k-1\end{pmatrix}\sum_{|\tau_1|=n-k}\sum_{|\tau_2|=k}\omega(\tau_1)\omega(\tau_2)
     \cdot\raisebox{-5pt}{{\scriptsize \begin{forest} for tree={grow'=90, parent anchor=center,child anchor=south, l=0cm,inner ysep=2pt,edge+=thick, s sep=0mm} [[$\tau_1$][$\tau_2$]] \end{forest}}},     
   \end{align*}
  where we shifted the $k$-summation label in the second term by one. Fifth, carefully considering of the $k=0$ case in the first term and the $k=n$ case in the second term in the last line above,  
  and using the identity,
  \begin{equation*}
    \begin{pmatrix} n-1\\ k\end{pmatrix}+\begin{pmatrix} n-1\\ k-1\end{pmatrix}=\begin{pmatrix} n\\ k\end{pmatrix},
  \end{equation*}
  we observe that,
  \begin{align*}
   B_\vee^{n+1}({\tiny \begin{forest} for tree={grow'=90, parent anchor=center,child anchor=center} [] \path[fill=black] (.anchor) circle[radius=1.5pt]; \end{forest}})
   =&\;\sum_{k=0}^{n} \sum_{|\tau_1|=n-k}\sum_{|\tau_2|=k}\begin{pmatrix} n\\ k\end{pmatrix}\omega(\tau_1)\omega(\tau_2)
   \cdot\raisebox{-5pt}{{\scriptsize \begin{forest} for tree={grow'=90, parent anchor=center,child anchor=south, l=0cm,inner ysep=2pt,edge+=thick, s sep=0mm} [[$\tau_1$][$\tau_2$]] \end{forest}}}\\
   =&\;\sum_{|\tau|=n+1}\omega(\tau)\cdot\tau,
  \end{align*}
  using Lemma~\ref{lemma:grafting}. This gives the result for $\ell=n+1$. \qed
\end{proof}
We now solve the initial value problem~\eqref{eq:coagtrees} for $\gf=\gf(t)$ in $\R\lb\mathbb T\rb$,
using Lemmas~\ref{lemma:grafting} and \ref{lemma:branchpower}.
Indeed, the solution to the Smoluchowski tree equation has three useful formulations. 
\begin{theorem}[Main result: solution]\label{thm:main}
  The solution $\gf$ to the initial value problem~\eqref{eq:coagtrees} is given by, 
  \begin{align}
    \gf&=\sum_{\tau\in\mathbb T\backslash\emptyset} \frac{t^{|\tau|}}{|\tau|!}\omega(\tau)\cdot\tau\label{eq:mainsoln}\\
    &\equiv\bigl(\exp(tB_\vee)\bigr)({\tiny \begin{forest} for tree={grow'=90, parent anchor=center,child anchor=center} [] \path[fill=black] (.anchor) circle[radius=1.5pt]; \end{forest}})\label{eq:expsoln}\\
    &\equiv(\id-t\hat B_\vee)^{-1}({\tiny \begin{forest} for tree={grow'=90, parent anchor=center,child anchor=center} [] \path[fill=black] (.anchor) circle[radius=1.5pt]; \end{forest}}).\label{eq:matchsoln}
  \end{align}
  Further, with data $\gf(0)={\tiny \begin{forest} for tree={grow'=90, parent anchor=center,child anchor=center} [] \path[fill=black] (.anchor) circle[radius=1.5pt]; \end{forest}}$,
  $\gf=\gf(t)$ solves the equation,
  \begin{equation}\label{eq:coagtreeslinear}
    \pa_t\gf=B_\vee(\gf). 
  \end{equation} 
\end{theorem}
\begin{proof}
  Using~\eqref{eq:gndef}, the first form of the solution can be re-written as,
  \begin{equation}\label{eq:gdefproof}
     \gf=\sum_{n\in\mathbb N_0}\frac{t^n}{n!}\gf_n.
  \end{equation}
  Now if we use Lemma~\ref{lemma:grafting}, we observe,
  \begin{equation*}
    \pa_t\gf=\sum_{n\in\mathbb N_0}\frac{t^n}{n!}\gf_{n+1}
    =\sum_{n\in\mathbb N_0}\frac{t^n}{n!}\sum_{k=0}^n\begin{pmatrix} n\\ k\end{pmatrix}
    \cdot\raisebox{-5pt}{{\scriptsize \begin{forest} for tree={grow'=90, parent anchor=center,child anchor=south, l=0cm,inner ysep=2pt,edge+=thick, s sep=0mm} [[$\gf_{n-k}$][$\gf_k$]] \end{forest}}}.
  \end{equation*}
  On the other hand using~\eqref{eq:gdefproof} again, we observe,
  \begin{equation*}
    \raisebox{-5pt}{{\scriptsize \begin{forest} for tree={grow'=90, parent anchor=center,child anchor=south, l=0cm,inner ysep=2pt,edge+=thick, s sep=1mm} [[$\gf$][$\gf$]] \end{forest}}}
    =\sum_{n\in\mathbb N_0}t^n\sum_{k=0}^n\frac{1}{(n-k)!k!}
    \cdot\raisebox{-5pt}{{\scriptsize \begin{forest} for tree={grow'=90, parent anchor=center,child anchor=south, l=0cm,inner ysep=2pt,edge+=thick, s sep=0mm} [[$\gf_{n-k}$][$\gf_k$]] \end{forest}}},
  \end{equation*}
  which matches our expression for $\pa_t\gf$ just above. The initial condition is naturally attained as $t\to0$.
  The second solution form shown follows from Lemma~\ref{lemma:branchpower}. We observe,
 \begin{equation*}
   \gf=\sum_{n\in\mathbb N_0}\frac{t^n}{n!}\gf_n
   =\sum_{n\in\mathbb N_0}\frac{t^n}{n!}B_\vee^n({\tiny \begin{forest} for tree={grow'=90, parent anchor=center,child anchor=center} [] \path[fill=black] (.anchor) circle[radius=1.5pt]; \end{forest}})
   =\bigl(\exp(tB_\vee)\bigr)({\tiny \begin{forest} for tree={grow'=90, parent anchor=center,child anchor=center} [] \path[fill=black] (.anchor) circle[radius=1.5pt]; \end{forest}}).
  \end{equation*}
 The third solution form shown follows from the identification $\hat B_\vee\coloneqq G\circ B_\vee$. For any $n\in\mathbb N_0$ we have,
 \begin{equation*}
   \hat B_\vee^n({\tiny \begin{forest} for tree={grow'=90, parent anchor=center,child anchor=center} [] \path[fill=black] (.anchor) circle[radius=1.5pt]; \end{forest}})
   =\frac{1}{n!}B_\vee^n({\tiny \begin{forest} for tree={grow'=90, parent anchor=center,child anchor=center} [] \path[fill=black] (.anchor) circle[radius=1.5pt]; \end{forest}}),
 \end{equation*}
 giving the result. Finally, using~\eqref{eq:gdefproof} and \eqref{eq:gndef} we observe,
 \begin{equation*}
  B_{\vee}(\gf)=\sum_{n\in\mathbb N_0}\frac{t^n}{n!}\cdot B_\vee(\gf_n)=\sum_{n\in\mathbb N_0}\frac{t^n}{n!}\cdot\gf_{n+1}=\pa_t\gf,
 \end{equation*}
 which also follows from the second solution form. \qed
\end{proof}
\begin{remark}[Grafting and branching equilibrium]
  An immediate consequence of Theorem~\ref{thm:main} is that the solution form $\gf$ given therein, satisfies,
  \begin{equation*}
    B_\vee(\gf)=\raisebox{-5pt}{{\scriptsize \begin{forest} for tree={grow'=90, parent anchor=center,child anchor=south, l=0cm,inner ysep=2pt,edge+=thick, s sep=1mm} [[$\gf$][$\gf$]] \end{forest}}}.
  \end{equation*}
  We can thus interpret the solution flow $\gf$ in Theorem~\ref{thm:main} as one for which, as time evolves, the processes of grafting and branching precisely match each other.
  Their actions are equivalent on such a flow, or exactly compatible.
\end{remark}
\begin{remark}[Exponential series solution]
  The solution form $\gf$ in Theorem~\ref{thm:main} is an exponential series of the form in Definition~\ref{def:expseries},
  with, for each $\tau\in\mathbb T\backslash\emptyset$: $\hat\ff_\tau=t^{|\tau|}$. 
\end{remark}

\begin{table}
  \caption{We list all the rooted planar binary trees of grade $4$. For each tree $\tau\in\mathbb T$ in the left column,
    we give its corresponding weight $\omega(\tau)$ in the second column, the total number of symmetries $\sigma(\tau)$ it possesses
    (see Remark~\ref{rmk:nonplanartrees}) in the third column and its word-coding in terms of levels in the final column.}
\label{table:treesof4}
\begin{center}
  \setlength{\baselineskip}{2\baselineskip}
\begin{tabular}{|c|c|c|c|}
\hline
$\phantom{\biggl|}$ tree $\phantom{\biggl|}$ & weight & symm.\/ & levels \\ \hline
&&&\\[-5pt]
{\tiny \begin{forest} for tree={grow'=90, parent anchor=center,child anchor=center, l=0cm,inner ysep=0cm,edge+=thick} [[][[][[][[][]]]]] \end{forest}} &
\raisebox{\depth}{{\scriptsize $\begin{pmatrix} 3 \!\!&\!\!2\!\!&\!\! 1\!\! &\!\! 0 \\ 0 \!\!&\!\!0 \!\!&\!\! 0 \!\!&\!\! 0\end{pmatrix}$}}& 3 & 1234\\
{\tiny \begin{forest} for tree={grow'=90, parent anchor=center,child anchor=center, l=0cm,inner ysep=0cm,edge+=thick} [[][[][[[][]][]]]] \end{forest}} &
\raisebox{\depth}{{\scriptsize $\begin{pmatrix} 3 \!\!&\!\!2\!\!&\!\! 1\!\! &\!\! 0 \\ 0 \!\!&\!\!0 \!\!&\!\! 1 \!\!&\!\! 0\end{pmatrix}$}}& 3 & 1243\\[5pt]
{\tiny \begin{forest} for tree={grow'=90, parent anchor=center,child anchor=center, l=0cm,inner ysep=0cm,edge+=thick} [[][[[][]][[][]]]] \end{forest}} &
\raisebox{\depth}{{\scriptsize $\begin{pmatrix} 3 \!\!&\!\!2\!\!&\!\! 0\!\! &\!\! 0 \\ 0 \!\!&\!\!1 \!\!&\!\! 0 \!\!&\!\! 0\end{pmatrix}$}}& 1 & 1323\\
{\tiny \begin{forest} for tree={grow'=90, parent anchor=center,child anchor=center, l=0cm,inner ysep=0cm,edge+=thick} [[][[[][[][]]][]]] \end{forest}} &
\raisebox{\depth}{{\scriptsize $\begin{pmatrix} 3 \!\!&\!\!2\!\!&\!\! 1\!\! &\!\! 0 \\ 0 \!\!&\!\!2 \!\!&\!\! 0 \!\!&\!\! 0\end{pmatrix}$}}& 3 & 1342\\
{\tiny \begin{forest} for tree={grow'=90, parent anchor=center,child anchor=center, l=0cm,inner ysep=0cm,edge+=thick} [[][[[[][]][]][]]] \end{forest}} &
\raisebox{\depth}{{\scriptsize $\begin{pmatrix} 3 \!\!&\!\!2\!\!&\!\! 1\!\! &\!\! 0 \\ 0 \!\!&\!\!2 \!\!&\!\! 1 \!\!&\!\! 0\end{pmatrix}$}}& 3 & 1432\\[5pt]
{\tiny \begin{forest} for tree={grow'=90, parent anchor=center,child anchor=center, l=0cm,inner ysep=0cm,edge+=thick} [[[][]][[][[][]]]] \end{forest}} &
\raisebox{\depth}{{\scriptsize $\begin{pmatrix} 3 \!\!&\!\!1\!\!&\!\! 0\!\! &\!\! 0 \\ 1 \!\!&\!\!0 \!\!&\!\! 0 \!\!&\!\! 0\end{pmatrix}$}}& 2 & 2123\\[5pt]
{\tiny \begin{forest} for tree={grow'=90, parent anchor=center,child anchor=center, l=0cm,inner ysep=0cm,edge+=thick} [[[][]][[[][]][]]] \end{forest}} &
\raisebox{\depth}{{\scriptsize $\begin{pmatrix} 3 \!\!&\!\!1\!\!&\!\! 0\!\! &\!\! 0 \\ 1 \!\!&\!\!1 \!\!&\!\! 0 \!\!&\!\! 0\end{pmatrix}$}}& 2 & 2132\\[5pt]
{\tiny \begin{forest} for tree={grow'=90, parent anchor=center,child anchor=center, l=0cm,inner ysep=0cm,edge+=thick} [[[][[][]]][[][]]] \end{forest}} &
\raisebox{\depth}{{\scriptsize $\begin{pmatrix} 3 \!\!&\!\!1\!\!&\!\! 0\!\! &\!\! 0 \\ 2 \!\!&\!\!0 \!\!&\!\! 0 \!\!&\!\! 0\end{pmatrix}$}}& 2 & 2312\\[5pt]
{\tiny \begin{forest} for tree={grow'=90, parent anchor=center,child anchor=center, l=0cm,inner ysep=0cm,edge+=thick} [[[[][]][]][[][]]] \end{forest}} &
\raisebox{\depth}{{\scriptsize $\begin{pmatrix} 3 \!\!&\!\!1\!\!&\!\! 0\!\! &\!\! 0 \\ 2 \!\!&\!\!1 \!\!&\!\! 0 \!\!&\!\! 0\end{pmatrix}$}}& 2 & 3212\\
{\tiny \begin{forest} for tree={grow'=90, parent anchor=center,child anchor=center, l=0cm,inner ysep=0cm,edge+=thick} [[[][[][[][]]]][]] \end{forest}} &
\raisebox{\depth}{{\scriptsize $\begin{pmatrix} 3 \!\!&\!\!2\!\!&\!\! 1\!\! &\!\! 0 \\ 3 \!\!&\!\!0 \!\!&\!\! 0 \!\!&\!\! 0\end{pmatrix}$}}& 3 & 2341\\
{\tiny \begin{forest} for tree={grow'=90, parent anchor=center,child anchor=center, l=0cm,inner ysep=0cm,edge+=thick} [[[][[[][]][]]][]] \end{forest}} &
\raisebox{\depth}{{\scriptsize $\begin{pmatrix} 3 \!\!&\!\!2\!\!&\!\! 1\!\! &\!\! 0 \\ 3 \!\!&\!\!0 \!\!&\!\! 1 \!\!&\!\! 0\end{pmatrix}$}}& 3 & 2431\\[5pt]
{\tiny \begin{forest} for tree={grow'=90, parent anchor=center,child anchor=center, l=0cm,inner ysep=0cm,edge+=thick} [[[[][]][[][]]][]] \end{forest}} &
\raisebox{\depth}{{\scriptsize $\begin{pmatrix} 3 \!\!&\!\!2\!\!&\!\! 0\!\! &\!\! 0 \\ 3 \!\!&\!\!1 \!\!&\!\! 0 \!\!&\!\! 0\end{pmatrix}$}}& 1 & 3231\\
{\tiny \begin{forest} for tree={grow'=90, parent anchor=center,child anchor=center, l=0cm,inner ysep=0cm,edge+=thick} [[[[][[][]]][]][]] \end{forest}} &
\raisebox{\depth}{{\scriptsize $\begin{pmatrix} 3 \!\!&\!\!2\!\!&\!\! 1\!\! &\!\! 0 \\ 3 \!\!&\!\!2 \!\!&\!\! 0 \!\!&\!\! 0\end{pmatrix}$}}& 3 & 3421\\
{\tiny \begin{forest} for tree={grow'=90, parent anchor=center,child anchor=center, l=0cm,inner ysep=0cm,edge+=thick} [[[[[][]][]][]][]] \end{forest}} &
\raisebox{\depth}{{\scriptsize $\begin{pmatrix} 3 \!\!&\!\!2\!\!&\!\! 1\!\! &\!\! 0 \\ 3 \!\!&\!\!2 \!\!&\!\! 1 \!\!&\!\! 0\end{pmatrix}$}}& 3 & 4321\\
\hline
\end{tabular}
\end{center}
\end{table}

\section{Coagulation examples}\label{sec:coagulation}
Herein we explore some of the ramifications of the binary tree solution expansions in the last section, and in particular,
make connections to cases with known explicit solutions. In this and the next section we are focused on the class of exponential
tree expansions of the form~\eqref{eq:exptreeseries} given in Definition~\ref{def:expseries}. In our coagulation application
we ally any such tree expansion with a positive function $\xi$ representing the generalised transform of the initial data. 
Hence an exponential tree series of the form~\eqref{eq:exptreeseries} in $\R\la\mathbb T\ra$ represents the following expansion
in terms of the generalised transform function $\xi$, namely,
\begin{equation}\label{eq:expseriesapp}
\ff=\sum_{\tau\in\mathbb T\backslash\emptyset} \frac{\omega(\tau)}{|\tau|!}\hat\ff_\tau\,\xi(\tau).
\end{equation}
Here, for any tree $\tau\in\mathbb T\backslash\emptyset$, the forms $\xi(\tau)$ are to be interpreted as outlined in the
introduction and at the beginning of Section~\ref{sec:trees}, so that for example,
\begin{equation*}
  \xi\Bigl(\raisebox{-4pt}{{\tiny \begin{forest} for tree={grow'=90, parent anchor=center,child anchor=center, l=0cm,inner ysep=0cm,edge+=thick} [[][[][]]] \end{forest}}}\Bigr)
  \equiv\xi\star(\xi\star\xi),
\end{equation*}
and so forth. Suppose there exist positive constants $c$ and $\Lambda<1$ such that for all $\tau\in\mathbb T\backslash\emptyset$,
we have,
\begin{equation}\label{eq:conv}
\|\hat\ff_\tau\,\xi(\tau)\|\leqslant c\,\Lambda^{|\tau|},
\end{equation}
where here, $\|\cdot\|$ represents the absolute value---in this
case of the real-valued coefficients $\hat\ff_\tau$ and functions $\xi(\tau)$. If condition~\eqref{eq:conv} holds,
the series~\eqref{eq:expseriesapp} is convergent as,
\begin{equation*}
  \|\ff\|\leqslant c\sum_{\tau\in\mathbb T\backslash\emptyset} \frac{\omega(\tau)}{|\tau|!}\,\Lambda^{|\tau|}
  \leqslant c\sum_{n\geqslant0}\Lambda^{n}.
\end{equation*}
Here we used that for a fixed grade $n$, the sum, over all trees $\tau$ of grade $n$, of $\omega(\tau)$ is $n!$.
We can prove this by induction. Assume the result to be true for $k=0,1,\ldots,n$. Recall the each of the
trees at grade `$n+1$' are constructed uniquely by the root grafting all the trees of grade `$k$' to those of grade `$n-k$'
for $k=0,1,\ldots,n$. Hence using the property of the weight character in Definition~\ref{def:weightchar}, we observe,
\begin{equation*}
  \sum_{k=0}^n\sum_{|\tau_1|=k}\sum_{|\tau_2|=n-k}\!\!\begin{pmatrix} n\\ k\end{pmatrix}\omega(\tau_1)\omega(\tau_2)
    =\sum_{k=0}^n\begin{pmatrix} n\\ k\end{pmatrix}k!\,(n-k)!, 
\end{equation*}
which equals $(n+1)!$, giving the result. Hence our algebra $\Hb=\Hb(\xi,\star)$ in Section~\ref{sec:trees},
denotes the \emph{class of convergent exponential tree series}, convergent with respect to the data $\xi$,
i.e.\/ those exponential tree series for which the condition~\eqref{eq:conv} is satisfied.

For the general frequency kernel case $K=K(y,z)$, the solution to the coagulation equation~\eqref{eq:quadnonassoc} with the data $\xi$,
represented by \eqref{eq:mainsoln} in Theorem~\ref{thm:main}, is given by,
\begin{equation}\label{eq:solnseriesapp}
\gf=\sum_{\tau\in\mathbb T\backslash\emptyset} \frac{t^{|\tau|}}{|\tau|!}\omega(\tau)\,\xi(\tau).
\end{equation}
We identify the exponential tree series coefficients in this case as $\hat\gf_\tau=t^{|\tau|}$.
This series is convergent provided $t<1$ and $\|\xi(\tau)\|\leqslant c$ for some $c>0$, for all $\tau\in\mathbb T\backslash\emptyset$.
This brief, general analysis here, for suitable kernels for which $\|\xi(\tau)\|\leqslant c$,
thus only establishes existence of such a solution \emph{locally} in time, indeed, for $t<1$.

Presently we consider the three classical coagulation cases of the constant, additive and multiplicative frequency
kernels. In preparation, as in Menon and Pego~\cite{MP}, we specialise the general transform $\mathcal H$ to the 
Bernstein transform. In this case the kernel of $\mathcal H$ is $h(s,x)\coloneqq1-\mathrm{e}^{-sx}$.
Thus, the Bernstein transform represents a desingularised Laplace transform. 
In fact the Bernstein transform exists for any positive Radon measure $\nu_t=\nu_t(\rd x)$ on $(0,\infty)$ with scalar real parameter $t$,
in the form, 
\begin{equation}\label{eq:mBernsteinT}
\gf(s;t)\coloneqq\int_0^\infty (1-\mathrm{e}^{-sx})\,\nu_t(\rd x).
\end{equation}
We assume the data $\nu_0$ is a positive Radon measure on $(0,\infty)$ and $\xi$ is the Bernstein transform of $\nu_0$.
Menon and Pego thus consider a weak formulation of \eqref{eq:abstract} for which,  
\begin{equation}\label{eq:evolweak}
  \bigl(\gf\star\gf\bigr)(s;t)=\iint\bigl(H(s,y,z)K(y,z)\bigr)\,\nu_t(\rd y)\,\nu_t(\rd z),
\end{equation}
where $H(s,y,z)\coloneqq\frac12(1-\mathrm{e}^{-sy})(1-\mathrm{e}^{-sz})$ and the double integral is over $[0,\infty)^2$.
For the precise details of the weak formulation setting, see Menon and Pego~\cite{MP}.
For this special form of $H$, we can give an alternative characterisation for $\gf\star\gf$ as follows.
Suppose we can expand $K$ in the following separable form, 
\begin{equation}\label{eq:sepform}
K(y,z)=\sum_{k,\ell\in\mathbb N_0}c_{k\ell}\,y^kz^\ell,
\end{equation}
for some constants $c_{k\ell}\geqslant0$, where $\mathbb N_0\coloneqq\mathbb N\cup\{0\}$. Since $K$ is symmetric, $c_{k\ell}=c_{\ell k}$.
Then with $\pa=\pa_s$, we have,   
\begin{equation*}
\gf\star\gf=\tfrac12\sum_{k,\ell\in\mathbb N_0}c_{k\ell}\,\bigl((-1)^k\pa^k\gf+M_k\bigr)\bigl((-1)^\ell\pa^\ell\gf+M_\ell\bigr),
\end{equation*}
where for $k\in\mathbb N_0$, the moments $M_k\coloneqq\int_0^\infty x^k\,g(x)\,\rd x$.
In this sum, when $k$ or $\ell$ are zero, the added $M_0$ term should be taken to be zero.
If the separable form expansion~\eqref{eq:sepform} is finite, then this
definition for the non-associative product `$\star$' could be computationally useful as
we only have to compute a finite number of moments. Indeed, this is the case in the classical
coagulation cases we now focus on.
\begin{example}[Constant kernel]\label{ex:constant}
Consider the constant $K=1$ frequency kernel case. The Bernstein transform $\gf=\gf(s;t)$ of $\nu_t$ satisfies,
\begin{equation*}
\pa_t\gf=-\tfrac12\gf^2.
\end{equation*}
The product `$\star$' is just the usual \emph{associative} real product, with the $-\frac12$ factor.
The solution form~\eqref{eq:matchsoln} in this case is given by
$\gf=(\id-t\hat B_\vee)^{-1}\circ\xi({\tiny \begin{forest} for tree={grow'=90, parent anchor=center,child anchor=center} [] \path[fill=black] (.anchor) circle[radius=1.5pt]; \end{forest}})$.
The action of the operator $\hat B_\vee$ in this associative case is simple. Indeed we observe that 
$\hat B_\vee\circ\xi({\tiny \begin{forest} for tree={grow'=90, parent anchor=center,child anchor=center} [] \path[fill=black] (.anchor) circle[radius=1.5pt]; \end{forest}})=-\frac12\,\xi^2$.
Then the action of $\hat B_\vee$ on this last form is to successively, additively replace each $\xi$ factor on the right by $-\frac12\,\xi^2$, giving
$(\hat B_\vee)^2\circ\xi({\tiny \begin{forest} for tree={grow'=90, parent anchor=center,child anchor=center} [] \path[fill=black] (.anchor) circle[radius=1.5pt]; \end{forest}})=\frac14\,\xi^3$.
Recall here that the grading operator is implicit in $\hat B_\vee$. We also deduce that
$(\hat B_\vee)^3\circ\xi({\tiny \begin{forest} for tree={grow'=90, parent anchor=center,child anchor=center} [] \path[fill=black] (.anchor) circle[radius=1.5pt]; \end{forest}})=-\frac18\,\xi^4$.
and so forth. Hence we observe that,
\begin{equation*}
  \gf=(\id-t\hat B_\vee)^{-1}\circ\xi({\tiny \begin{forest} for tree={grow'=90, parent anchor=center,child anchor=center} [] \path[fill=black] (.anchor) circle[radius=1.5pt]; \end{forest}})
  =\xi\bigl(1+\tfrac12t\xi\bigr)^{-1},
\end{equation*}
which is the solution in this constant kernel case.
\end{example}
\begin{example}[Additive/multiplicative case]\label{ex:addmult}
The additive $K=y+z$ and multiplicative $K=yz$ frequency kernel cases
can be considered together for the following reason. The modified Bernstein transform $\gf^\ast$ of $\nu_t$
is similar to the Bernstein transform except that the kernel $h(s,x)\coloneqq x(1-\mathrm{e}^{-sx})$.
Respectively in the additive and multiplicative cases, the Bernstein transform $\gf$ of $\nu_t$
and the modified Bernstein transform $\gf^\ast$ of $\nu_t$, satisfy,
\begin{equation*}
  \pa_t\gf=\gf\pa\gf-\gf \qquad\text{and}\qquad \pa_t\gf^\ast=\gf^\ast\pa\gf^\ast.
\end{equation*}
In the additive case we have normalised the constant first moment $M_1=1$.
If we set $\hf\coloneqq\mathrm{e}^{t}\gf$, then $\pa\hf=\mathrm{e}^{-t}\hf\pa\hf$.
If we set $\kappa\coloneqq1-\mathrm{e}^{-t}$ and $\hf^\ast(s;\kappa)\coloneqq\hf(s,t)$, then $\hf^\ast$ satisfies,
$\pa_\kappa\hf^\ast=\hf^\ast\pa\hf^\ast$. Hence it is sufficient to focus on the case,
\begin{equation}\label{eq:suffcase}
  \pa_t\gf=\gf\pa\gf,
\end{equation}
with $\gf(s;0)=\xi(s)$. It is well-known that for this case, corresponding to the multiplicative frequency kernel case for $\gf^\ast$,
that after a suitable normalisation, the gelation time occurs at $t=1$. Herein we only consider the solution up to gelation,
though a solution exists beyond that time, see item (ii) in Section~\ref{sec:discussion}. Note that for the rescaled additive case for $\hf^\ast$,
the gelation time corresponds to $\kappa=1$ which translates back to $t=\infty$. See Menon and Pego~\cite{MP} for more details.
The solution to the inviscid Burgers equation~\eqref{eq:suffcase} is uniquely obtained up to $t=1$ via characteristics as,
\begin{equation}\label{eq:suffcasesoln}
     \gf\circ s=\xi\circ(\id-t\xi)^{\circ(-1)}\circ s, 
\end{equation}
where the centre term is a compositional inverse. 
For the equation~\eqref{eq:suffcase} with solution $\gf$ and data $\xi$, we have,
\begin{equation*}
\gf\star\gf=\gf\pa\gf,
\end{equation*}
which is non-associative. We know the solution form in this context is any of the forms~\eqref{eq:mainsoln}--\eqref{eq:matchsoln} in Theorem~\ref{thm:main},
there is a close visual match to \eqref{eq:matchsoln} in particular.
By uniqueness the solution forms~\eqref{eq:suffcasesoln} and~\eqref{eq:mainsoln} are one and the same---one could simply iterate~\eqref{eq:suffcase}
as we did for~\eqref{eq:quadnonassoc}.
\end{example}
\begin{remark}[Explicit match]\label{rmk:explicit}
  We can be more explicit about the solution match in Example~\ref{ex:addmult} as follows.
  If we set $\hf\coloneqq(\id-t\xi)^{\circ(-1)}$, then the solution~\eqref{eq:suffcasesoln} can be expressed in
  the form $\gf=\xi\circ\hf$ with $\hf$ satisfying $\hf=\id+t\,\xi\circ\hf$, or more explicitly, $\hf\circ s=s+t\,\xi\circ h\circ s$.
  Naturally we have $\gf_0=\xi$ and $\hf_0=\id$. Note that $\pa_t\hf=\xi\circ\hf+t\,(\pa_t\hf)(\pa\xi\circ\hf)$ or
  equivalently $\pa_t\hf=(\xi\circ\hf)(\id-t\,\pa\xi\circ\hf)^{-1}$, where the second factor is a reciprocal. Thus, since $\pa_t\gf=(\pa_t\hf)(\pa\xi\circ\hf)$,
  we observe that,
  \begin{equation*}
   \pa_t\gf=(\xi\circ\hf)\bigl(\pa\xi\circ\hf+t\,(\pa\xi\circ\hf)^2+t^2\,(\pa\xi\circ\hf)^3+\cdots\bigr).
  \end{equation*}
  We can use this form to compute the Taylor series expansion for $\gf$ in powers of $t$. Thus for example, we observe,
  \begin{align*}
    \pa_t\gf\vert_{t=0}=&\;\xi\pa\xi,\\
    \pa_t^2\gf\vert|_{t=0}=&\;\Bigl((\xi\circ\hf)(\pa\xi\circ\hf)^2\\
                          &\;+(\xi\circ\hf)^2\bigl(\pa^2\xi\circ\hf+(\pa\xi\circ\hf)^2\bigr)+\mathcal O(t)\Bigr)\vert_{t=0}\\
                         =&\;(\xi\pa\xi)\pa\xi+\xi\pa(\xi\pa\xi),
  \end{align*}
  and so forth.
\end{remark}
\begin{remark}[Rank-one analytic Grassmannian]
  In Doikou \textit{et al.\/} \cite{DMSW-coagulation}, we show that the additive and multiplicative kernel
  cases can be described in terms of the Fa\`a di Bruno composition algebra of analytic exponential power series.
  See Figueroa \textit{et al.\/} \cite{FG-BV} and Gessel~\cite{Gessel} for more details.
  Indeed, in these cases, inspired by the result of Byrnes~\cite{Byrnes} and Byrnes and Jhemi~\cite{BJ} on Riccati partial differential equations,
  in Doikou \textit{et al.\/} \cite{DMSW-coagulation}, we show that the solution $\gf$ has a representation as a rank-one analytic Grassmannian flow.
  See Section~\ref{sec:discussion} for more details on Grassmannian flows and the extension to the general frequency kernel case.
\end{remark}
\begin{remark}[Non-planar trees]\label{rmk:nonplanartrees}
  With the exception of Examples~\ref{ex:constant} and \ref{ex:addmult}, and Remark~\ref{rmk:explicit}, for full generality,
  we assume throughout that the non-associative product `$\star$' is non-commutative, and thus the rooted binary trees we consider are planar. 
  In Smoluchowski's equation~\eqref{eq:Smoluchowski} the fields are scalar and thus the resulting product `$\star$' is commutative.
  This means that many of the terms $\xi(\tau)$, for trees $\tau$ of the same grade, are equivalent/collapse together.
  Indeed, we can represent any exponential tree series expansion on the set of non-planar binary trees as opposed to planar ones.
  Thus, for example, commutativity of `$\star$' means that we do not distinguish between the trees,
  \begin{equation*}
    {\tiny \begin{forest} for tree={grow'=90, parent anchor=center,child anchor=center, l=0cm,inner ysep=0cm,edge+=thick} [[][[][]]] \end{forest}}
    \qquad\text{and}\qquad
    {\tiny \begin{forest} for tree={grow'=90, parent anchor=center,child anchor=center, l=0cm,inner ysep=0cm,edge+=thick} [[[][]][]] \end{forest}},
  \end{equation*}
  and the only two distinct rooted non-planar binary trees of grade $3$ are,
  \begin{equation*}
    {\tiny \begin{forest} for tree={grow'=90, parent anchor=center,child anchor=center, l=0cm,inner ysep=0cm,edge+=thick} [[][[][[][]]]] \end{forest}}\qquad\text{and}\qquad
    {\tiny \begin{forest} for tree={grow'=90, parent anchor=center,child anchor=center, l=0cm,inner ysep=0cm,edge+=thick} [[[][]][[][]]] \end{forest}}.
  \end{equation*}
  We can think of the set of rooted non-planar binary trees as a set of equivalence classes of rooted planar binary trees,
  where we choose a single representative for the classes of planar binary trees that are equivalent in the non-planar setting.
  For example, in the case of the two grade $2$ planar trees shown above, we might choose the left one to be the representative.
  We have given two possible representatives in the grade $3$ case above. In general for example, we might choose the representative
  to be the planar binary tree with the lowest word-order code---see Section~\ref{sec:simulation}. We can determine which planar
  trees are equivalent in the non-planar setting by systematically, successively `twisting' (or `rotating') all the vertices present in a given planar tree,
  to see, if by doing so, they can be transformed into our given representative. For example, for the right-hand grade $2$ planar above,
  twisting the top left vertex does not alter the planar tree---indeed this is true for any free vertex ends of any planar trees
  (i.e.\/ vertices with no further higher attachments). However, if we twist/rotate the bottom vertex, we obtain the grade $2$ tree on the left.
  In the case of the grade $3$ planar tree on the right, we see that twisting any of the vertices it contains, does not alter it.
  However if we examine all the other trees of grade $3$ in Table~\ref{table:treesupto3}, we see that we can twist the vertices at
  the bottom as well as the first level up, to transform them into the tree of grade $3$ shown on the left above. A careful examination
  of Table~\ref{table:treesof4} reveals that there are three non-planar trees of grade $4$, a further quick enumeration reveals there are six at grade $5$, and so forth.
  This substantial reduction in the number of trees required to represent the exponential tree series solution \eqref{eq:mainsoln} in Theorem~\ref{thm:main},
  has important implications for numerical approximations, as we discuss in Section~\ref{sec:simulation}.
  In Tables~\ref{table:treesupto3} and \ref{table:treesof4} we give a column entitled `symm.' which indicates the total number
  of symmetries $\sigma(\tau)$ the \emph{planar} trees $\tau\in\mathbb T$ shown have. Thus for example, the trees of grade $2$ in
  Table~\ref{table:treesupto3}, can each be transformed into each other in the non-planar setting by vertex twisting.
  Thus for either of these trees $\tau$, each generates $2^1$ trees (including itself) in the non-planar setting and we set $\sigma(\tau)=1$.
  The number of non-planar trees each generates is $2^{\sigma(\tau)}$. Now consider the trees of grade $3$ in Table~\ref{table:treesupto3}.
  The middle tree in the list shown cannot be transformed into any other and thus $\sigma(\tau)=0$ and the number of non-planar trees, including itself,
  it can generate is $2^{\sigma(\tau)}=1$. The four other trees $\tau$ of grade $3$ shown can each be transformed into each other,
  and so for each of them, $\sigma(\tau)=2$, and the number of non-planar trees each can generate by vertex twisting, including themselves, is $2^{\sigma(\tau)}=4$.
  The first tree $\tau$ in Table~\ref{table:treesof4} has $\sigma(\tau)=3$ as, including itself, it can generate $8$ non-planar binary tree copies.
  And so forth. 
\end{remark}
\begin{remark}[Positivity]\label{rmk:positivity}
  An important property of solutions to Smoluchowski's equation~\eqref{eq:Smoluchowski} is that for positive initial data,
  the solutions should be positive thereafter while they exist. In the case of the constant, additive and multiplicative frequency kernel cases,
  Menon and Pego~\cite{MP} establish positive Radon measure-valued solutions provided the initial measure is positive.
  This achieved by establishing that the Bernstein transform of the measure solution, or its derivative, is completely monotone in the sense of Feller~\cite[XIII.4]{Feller}.
  Menon and Pego's results cover Examples~\ref{ex:constant} and \ref{ex:addmult} above. We do not pursue positivity in this sense further here,
  however a valuable future investigation would be to determine or characterise the conditions under which the general solution~\eqref{eq:mainsoln} for $\gf$
  possesses complete monotonicity or a similar property guaranteeing positive Radon measure-valued solutions. 
\end{remark}
\begin{remark}[Well-posedness]\label{rmk:existence}
  Though we have shown, that if condition~\eqref{eq:conv} holds, there exists a convergent exponential tree series solution,
  we have not studied in detail when condition~\eqref{eq:conv} does hold. For example, as mentioned in Example~\ref{ex:addmult},
  it is well-known, see Menon and Pego~\cite{MP}, that in the multiplicative case a positive Radon measure-valued solution exists up to the gelation time
  which is governed by the initial second moment $M_2(0)$, which we can normalise to be one. This result can be found in McLeod~\cite{McLeod1962a,McLeod1962b,McLeod1962c}.
  Comprehensive accounts of existing well-posedness results for the Smoluchowski equation can be found in da Costa~\cite{daC} and Dubovskii~\cite{Dubovskii}.
  They include well-posedness for the case of a frequency kernel $K(y,z)\leqslant k_0(1+y+z)$ for some constant $k_0>0$, and
  a striking result by Carr and Da Costa~\cite{CdaC} establishing instantaneous gelation if there exist constants
  $\alpha$, $\beta$ with $\beta>\alpha>1$ such that $k_0\,(y^\alpha+z^\beta)\leqslant K(y,z)\leqslant k_1\,(yz)^\beta$ for some positive constants $k_0$ and $k_1$.
  Further general rigorous well-posedness results are given in Escobedo \textit{et al.\/} \cite{EMP} and Escobedo \textit{et al.\/} \cite{ELMP} for coagulation-fragmentation models.
  Probabilistic approaches to solutions include that by Deaconu and Tanr\'e~\cite{DT} for the three classical kernels, and the characterisation
  of `eternal' solutions, existing for $t\in\R$, by Bertoin~\cite{Bertoin} for the additive kernel.
  Such well-posedness results are very kernel-specific. The frequency kernel $K$ is an integral part of the product `$\star$' and thus
  solution series expansion~\eqref{eq:solnseriesapp} through the terms $\xi(\tau)$. A general investigation into
  well-posedness results that can be established through the analysis of the terms $\xi(\tau)$ in~\eqref{eq:solnseriesapp} is a definitive worthwhile future endeavour.
  For example, we know from Example~\ref{ex:addmult}, that if $\gf$ is the solution to the additive kernel case, 
  then $\hf^\ast(s;t)=(1-t)^{-1}\gf(s;-\log(1-t))$ is the solution to the multiplicative kernel case. In this instance, in the multiplicative case,
  we observe there must be a natural renormalisation of the solution series expansion~\eqref{eq:solnseriesapp} generating a pre-factor $(1-t)^{-1}$,
  thus explicitly exposing the gelation time.
\end{remark}

\section{Numerical simulation method}\label{sec:simulation}
There are two parts to this section. In the first part we show how the solution flow~\eqref{eq:mainsoln}--\eqref{eq:matchsoln} 
can be formulated as a linearised flow. In the second part we establish a new practical numerical simulation method based
on the binary tree expansion solution~\eqref{eq:mainsoln}.

First, focusing on a linearised flow formulation, we derive two equivalent prescriptions. 
Though not strictly required here, it is useful to introduce the budding operator. We discuss it more generally in item (v)
in Section~\ref{sec:discussion}. Here it is a re-interpretation of the branching operator. 
\begin{definition}[Budding operator]\label{def:budding}
  For any $\tau^\prime\in\mathbb T$, we define the \emph{budding operator} $\mathcal B_{\tau^\prime}\colon\R\la\mathbb T\ra\to\R\la\mathbb T\ra$
  as the operator that acts on any tree $\tau\in\mathbb T$ by successively, additively performing the following operation to
  each free end of $\tau$. The operator $\mathcal B_{\tau^\prime}$ extends the existing free end to a new branch on the right,
  and attaches, via a new branch, the tree $\tau^\prime$ on the left. So a new vertex is created where the free end was, and the
  left branch of the new vertex has $\tau^\prime$ attached, while the right branch is a new free end. 
\end{definition}
\begin{example}
For example, for any $\tau^\prime\in\mathbb T$ we have, 
\begin{equation*}
\mathcal B_{\tau^\prime}(\bullet)=\raisebox{-1pt}{{\tiny \begin{forest} for tree={grow'=90, parent anchor=center,child anchor=south, l=0cm,inner ysep=1pt,edge+=thick, s sep=0mm} [[$\tau^\prime$][]] \end{forest}}}
\quad\text{and}\quad 
\mathcal B_{\tau^\prime}\bigl(\raisebox{-3pt}{{\tiny \begin{forest} for tree={grow'=90, parent anchor=center,child anchor=center, l=0cm,inner ysep=0pt,edge+=thick, s sep=1mm} [[][]] \end{forest}}}\bigr)
=\raisebox{-6pt}{{\tiny \begin{forest} for tree={grow'=90, parent anchor=center,child anchor=center, l=0cm,inner ysep=0mm,edge+=thick} [[[$\tau^\prime\,\,\,\,\,\,$][$\phantom{\tau}$]][]] \end{forest}}}
+\raisebox{-6pt}{{\tiny \begin{forest} for tree={grow'=90, parent anchor=center,child anchor=center, l=0cm,inner ysep=0mm,edge+=thick} [[][[$\tau^\prime\,\,\,\,\,\,$][$\phantom{\tau}$]]] \end{forest}}}.
\end{equation*}
\end{example}
Importantly, we observe that $\mathcal B_{\bullet}\equiv B_\vee$. Further, we can extend $\mathcal B_{\tau^\prime}$ linearly so that
$\mathcal B_{a\cdot\tau_1^\prime+b\cdot\tau_2^\prime}=a\cdot\mathcal B_{\tau_1^\prime}+b\cdot\mathcal B_{\tau_2^\prime}$ for any $a,b\in\R$.
As in Example~\ref{ex:adapted}, we set $\hat{\mathcal B}\coloneqq G\circ\mathcal B$. Naturally we also have $\hat{\mathcal B}_{\bullet}\equiv\hat B_\vee$.
We now present a linearised flow prescription in the context of $\R\la\mathbb T\ra$. 
\begin{prescription}[Linearised flow]\label{prescription:linearisedflow}
  Suppose $\pf\in\R\la\mathbb T\ra$, the linear operator $\mathcal Q\colon\R\la\mathbb T\ra\to\R\la\mathbb T\ra$,
  and the exponential tree series $\gf\in\R\la\mathbb T\ra$ satisfy,
  $\pf(0)={\tiny \begin{forest} for tree={grow'=90, parent anchor=center,child anchor=center} [] \path[fill=black] (.anchor) circle[radius=1.5pt]; \end{forest}}$,
  and $\mathcal Q(0)=\id$, and the linear system of equations,
  \begin{equation}\label{eq:linearisedflow}
    \pa_t\pf=0,\qquad \pa_t\mathcal Q=-\hat{\mathcal B}_{\pf}\qquad\text{and}\qquad\pf=\mathcal Q\gf.
  \end{equation}
\end{prescription}
We can solve the first two equations giving,
\begin{equation}\label{eq:linearsolns}
\pf={\tiny \begin{forest} for tree={grow'=90, parent anchor=center,child anchor=center} [] \path[fill=black] (.anchor) circle[radius=1.5pt]; \end{forest}}
\qquad\text{and}\qquad 
\mathcal Q=\bigl(\id-t\hat{\mathcal B}_{{\tiny \begin{forest} for tree={grow'=90, parent anchor=center,child anchor=center} [] \path[fill=black] (.anchor) circle[radius=1.5pt]; \end{forest}}}\bigr)\equiv\bigl(\id-t\hat B_\vee\bigr).
\end{equation}
Naturally, we observe that the linearised flow Prescription~\ref{prescription:linearisedflow}, with \eqref{eq:linearsolns} in hand,
is just a re-writing of the solution flow~\eqref{eq:mainsoln}--\eqref{eq:matchsoln}, however the interpretation of 
Prescription~\ref{prescription:linearisedflow} is useful.
Recall that `${\tiny \begin{forest} for tree={grow'=90, parent anchor=center,child anchor=center} [] \path[fill=black] (.anchor) circle[radius=1.5pt]; \end{forest}}$' in $\mathbb T$
represents the data function $\xi=\xi(s)$ in the coagulation context. Note that $\mathcal Q=\bigl(\id-t\hat B_\vee\bigr)$ is a linear endomorphism on $\R\la\mathbb T\ra$,
and thus the relation $\pf=\mathcal Q\gf$ in~\eqref{eq:linearisedflow} is indeed a linear equation for $\gf$, whose solution is \eqref{eq:matchsoln}.
To interpret this further we need to impose an \emph{ordering} on the set of rooted binary trees $\mathbb T$. 
Recall Tables~\ref{table:treesupto3} and \ref{table:treesof4} and focus on the final column labelled `levels'.
Roughly, we observe from the ordered listing of the trees shown in the tables, that we order them according to
more foliage on the right---this being also interpreted at each vertex---through to more foliage on the left. 
The words in the final column `levels' provide a coding for the trees represented, and, can either be read left to right,
or from the root up. Each level refers to the vertical position of the vertices of the trees shown.
\begin{procedure}[Tree order]\label{proc:treeorder}
An ordering for the set of rooted planar binary trees at each grade can be established as follows. 
Let us focus on the grade $4$ trees in Table~\ref{table:treesof4}. Interpreting left to right, consider the
tree labelled $3212$. Scanning from the left, the first vertex we encounter is at level $3$, the next at level $2$, then
level $1$ and then level $2$ again. For the tree labelled $2341$, the first vertex we encounter scanning from the left
is at level~$2$. We then follow those connected vertices that head upwards, so we get $34$, before, having exhausted that branch,
we proceed to the next vertex at level $1$. And so forth. The length of the word-code naturally corresponds to the grade of the tree.
The interpretation of the `levels' word-coding from the root up, proceeds as follows. Consider the tree labelled $3212$.
The $1$ indicates the first bottom vertex. Each word-code of any length corresponding to a tree with at least one vertex, has only a single $1$, naturally.
The two $2$'s straddling the `$1$' indicate that there are vertices attached to each of the ends of the bottom vertex.
Each word-code of any length has either none, one or two $2$'s. The case of no `$2$' means that the tree in question is just the grade $1$ tree.
The case of only one `$2$' means that only one branch of the root vertex has a vertex attached. If the `$2$' occurs somewhere to the right of the `$1$' in the word-code,
the vertex is attached to the right branch. If the `$2$' occurs somewhere to the left of the `$1$', the vertex is attached to the left branch.
In the word-code $3212$, the single $3$---there can only be one further digit for a grade $4$ tree---means that there is a single vertex at level $3$.
That the `$3$' is to the left of the left `$2$' means that the vertex is attached to the left branch of the left vertex at level $2$.
As another example consider the tree corresponding to the word-code $2431$. The `$1$' denotes the root vertex.
The single `$2$' to the left of the `$1$' indicates there is only one vertex attached to the left branch of the root vertex.
The single `$3$' to the right of the `$2$' indicates there is a single vertex attached to the right branch of the single vertex at level $2$.
The single `$4$' to the left of the `$3$', indicates there is a single vertex at level $4$ attached to the left branch of the vertex at level $3$.
And so forth. The set of rooted planar binary trees is then ordered by the numerical ordering, within the set of integers, of the word-codes. 
See the ordered lists in Tables~\ref{table:treesupto3} and \ref{table:treesof4}, and Example~\ref{ex:wordordering} below.
\end{procedure}
\begin{example}\label{ex:wordordering}
  Two further illustrative examples of the word-coding are as follows,
  \begin{equation*}
    {\tiny \begin{forest} for tree={grow'=90, parent anchor=center,child anchor=center, l=0cm,inner ysep=0cm,edge+=thick} [[[][]][[[][]][[][]]]] \end{forest}}=21323\quad\text{and}\quad
    {\tiny \begin{forest} for tree={grow'=90, parent anchor=center,child anchor=center, l=0cm,inner ysep=0cm,edge+=thick} [[[[][]][[[][]][[][[][]]]]][[][]]] \end{forest}}=32434512.
  \end{equation*}
\end{example}
\begin{remark}
  The word-code order of rooted planar binary trees above is closely related to their coding by permutations given in the original paper by Loday and Ronco~\cite{LR}.
  Also see Aguiar and Sottile~\cite{AguiarSottile}, Arcis and M\'arquez~\cite{Arcis} and Chatel and Pilaud~\cite{CP}.
\end{remark}
Consider the class of exponential tree series $\ff$ of the form~\eqref{eq:exptreeseries}, which we henceforth restrict ourselves to.
We order the terms in any such series according to the tree ordering outlined in Procedure~\ref{proc:treeorder}.
With this in hand, we can represent any such series by the vector of the coefficients $\hat\hf_\tau\coloneqq\omega(\tau)\,\hat\ff_\tau$,
associated with each tree $\tau\in\mathbb T\backslash\emptyset$, contained therein. In other words, we represent any such $\ff$ by the vector,
\begin{equation}\label{eq:expvector}
\hat\ff=(\hat\hf_{0};\hat\hf_1;\hat\hf_{12},\hat\hf_{21};\hat\hf_{123},\hat\hf_{132},\hat\hf_{212},\hat\hf_{231},\hat\hf_{321};\hat\hf_{1234},\cdots)^{\mathrm{T}},
\end{equation}
where the sub-indices are the word-codings of the corresponding trees. We use semi-colons to separate coefficients associated with trees of the same grade.
The basis elements are the corresponding trees $\tau$ with the reciprocal of the grade factorial factors `$1/|\tau|!$'. 
Let us now consider the action of the branching operator $B_\vee$ from Definition~\ref{def:branchingoperator} in the context of such coefficient vectors $\hat\ff$.
Since $B_\vee$ successively, additively attaches a branch 
`\raisebox{-1pt}{{\tiny \begin{forest} for tree={grow'=90, parent anchor=center,child anchor=center, l=0cm,inner ysep=0cm,edge+=thick} [[][]] \end{forest}}}'
to each free end of any tree $\tau\in\mathbb T$, its action on the word-coding of binary trees can be described as follows. For trees of lower grades, we observe:
\begin{align*}
  B_\vee(1)=&\;12+21,\\
  B_\vee(12+21)=&\;\prescript{\uparrow}{}{1}\,\vert\,\prescript{\uparrow}{}{2}^{\uparrow}+\prescript{\uparrow}{}{2}^{\uparrow}\,\vert\,1^{\uparrow}\\
  =&\;123+132+212+212+231+321.  
\end{align*}
Here, in the second example, we separate the digits of the word-code by a `$\vert$', and the arrows `$\uparrow$' indicate where a free branch of the
corresponding tree is and thus where we can successively, additively add a next level integer. Consider, for example, the tree represented by the word-code $12$.
This contains a vertex at level $1$ and attached to its right branch is a level $2$ vertex.
We can always attach a vertex `\raisebox{-1pt}{{\tiny \begin{forest} for tree={grow'=90, parent anchor=center,child anchor=center, l=0cm,inner ysep=0cm,edge+=thick} [[][]] \end{forest}}}'
to either end (left or right) of any trees as these are always free branches. For the word-code $12$ let us scan free branches right-to-left. The form $\prescript{\uparrow}{}{2}^{\uparrow}$ indicates
we can attach a vertex `\raisebox{-1pt}{{\tiny \begin{forest} for tree={grow'=90, parent anchor=center,child anchor=center, l=0cm,inner ysep=0cm,edge+=thick} [[][]] \end{forest}}}'
at level $3$ to the right branch of the level $2$ vertex, giving $123$. And since the integer digit to the immediate left of the $2$ is less in integer-order than the $2$,
there must be a free left-branch of the vertex at level $2$ in $12$, thus generating $132$. Then the form $\prescript{\uparrow}{}{1}$ indicates, since the integer digit to
the immediate right is greater than $1$ in integer-order, that there is not a free branch to the right for this level $1$ vertex. However, there is free branch on the left to which
we can add a level $2$ vertex, generating the tree $212$. The interpretation of the form $\prescript{\uparrow}{}{2}^{\uparrow}\,\vert\,1^{\uparrow}$ is now straightforward.
Note, starting with $12+21$ in tree order, when we consider attaching vertices to free branches of $12$ and $21$, we can do so
successively, additively from right-to-left. The order of the trees shown as the result of applying $B_\vee$ to `$12+21$' above is then a sum of
trees in tree-order. Though scanning right-to-left generates the trees at the next grade in tree order for a given word-code, this property does not generally
propagate across different word-codes in tree order. Indeed, let us consider some more examples, as follows. We observe,
\begin{align*}
  B_\vee(123)=&\;\prescript{\uparrow}{}{1}\,\vert\,\prescript{\uparrow}{}{2}\,\vert\,\prescript{\uparrow}{}{3}^{\uparrow}=1234+1243+1323+2123,\\
  B_\vee(132)=&\;\prescript{\uparrow}{}{1}\,\vert\,\prescript{\uparrow}{}{3}^{\uparrow}\,\vert\,2^{\uparrow}=1323+1342+1432+2132,\\ 
  B_\vee(212)=&\;\prescript{\uparrow}{}{2}^{\uparrow}\,\vert\,1\,\vert\,\prescript{\uparrow}{}{2}^{\uparrow}=2123+2132+2312+3212,\\
  B_\vee(231)=&\;\prescript{\uparrow}{}{2}\,\vert\,\prescript{\uparrow}{}{3}^{\uparrow}\,\vert\,1^{\uparrow}=2312+2341+2431+3231,\\
  B_\vee(321)=&\;\prescript{\uparrow}{}{3}^{\uparrow}\,\vert\,2^{\uparrow}\,\vert\,1^{\uparrow}=3212+3231+3421+4321.         
\end{align*}
For the example $32434512$ we had above, we observe,
\begin{equation*}
  B_\vee(32434512)=
  \prescript{\uparrow}{}{3}^{\uparrow}\,\vert\,2\,\vert\,\prescript{\uparrow}{}{4}^{\uparrow}\,\vert\,3\,\vert\,
  \prescript{\uparrow}{}{4}\,\vert\,\prescript{\uparrow}{}{5}^{\uparrow}\,\vert\,1\,\vert\,\prescript{\uparrow}{}{2}^{\uparrow}.
\end{equation*}
Note that as we consider each digit in $32434512$, free branches only occur when neighbouring digits are smaller in integer-order, as
indicated by the $\uparrow$'s above.

Now consider an exponential series represented by the vector~\eqref{eq:expvector}.
The action of the graded branching operator $\hat B_\vee$ on $\hat\ff$, i.e.\/ $\hat B_\vee(\hat\ff)$, is given by,
\begin{align*}
  \hat B_\vee&\bigl((\hat\hf_{0};\hat\hf_1;\hat\hf_{12},\hat\hf_{21};\hat\hf_{123},\hat\hf_{132},\hat\hf_{212},\hat\hf_{231},\hat\hf_{321};\hat\hf_{1234},\cdots)^{\mathrm{T}}\bigr)\\ 
  &=(0;\hat\hf_{0};\hat\hf_1,\hat\hf_1;\hat\hf_{12},\hat\hf_{12},\hat\hf_{12}+\hat\hf_{21},\hat\hf_{21},\hat\hf_{21};\hat\hf_{123},\cdots)^{\mathrm{T}}.
\end{align*}
We can represent the action of $\hat B_\vee$ on $\hat\ff$ as a lower triangular matrix operator on vectors $\hat\ff$ as follows,
\begin{equation*}
\hat B_{\mathrm{m}}\,\hat\ff=\begin{pmatrix}
  0 & 0 & 0 & 0 & 0 & 0 & 0 & \cdots \\
  1 & 0 & 0 & 0 & 0 & 0 & 0 & \cdots \\
  0 & 1 & 0 & 0 & 0 & 0 & 0 & \cdots \\
  0 & 1 & 0 & 0 & 0 & 0 & 0 & \cdots \\  
  0 & 0 & 1 & 0 & 0 & 0 & 0 & \cdots \\
  0 & 0 & 1 & 0 & 0 & 0 & 0 & \cdots \\
  0 & 0 & 1 & 1 & 0 & 0 & 0 & \cdots \\
  0 & 0 & 0 & 1 & 0 & 0 & 0 & \cdots \\
  0 & 0 & 0 & 1 & 0 & 0 & 0 & \cdots \\
  0 & 0 & 0 & 0 & 1 & 0 & 0 & \cdots \\
  \vdots & \vdots & \vdots & \vdots & \vdots & \vdots & \vdots & \vdots 
\end{pmatrix}
\begin{pmatrix}
  \hat\hf_{0}\\ \hat\hf_1\\ \hat\hf_{12}\\ \hat\hf_{21}\\ \hat\hf_{123}\\ \hat\hf_{132}\\ \hat\hf_{212}\\ \hat\hf_{231}\\ \hat\hf_{321}\\ \hat\hf_{1234} \\ \vdots 
\end{pmatrix},
\end{equation*}
where $\hat B_{\mathrm{m}}$ denotes the infinite matrix form of the graded branching operator. On the set of vectors $\hat\ff$ corresponding to the
class of exponential tree series, $\hat B_{\mathrm{m}}$ and all its powers are well-defined. Thus so is $(\id-t\hat B_{\mathrm{m}})^{-1}$ for sufficiently small
values of $t$. In particular, note that in the linearised flow in Prescription~\ref{prescription:linearisedflow}, 
the vector of coefficients corresponding to $\pf$ is $\hat\pf=(1;0;0,0;0,\ldots)^{\mathrm{T}}$. The successive action of $\hat B_{\mathrm{m}}$ is given by,
\begin{align*}
  \hat B_{\mathrm{m}}\,(1;0;0,0;0,\ldots)^{\mathrm{T}}&=(0;1;0,0;0,\ldots)^{\mathrm{T}},\\
  \hat B_{\mathrm{m}}^2(1;0;0,0;0,\ldots)^{\mathrm{T}}&=(0;0;1,1;0,\ldots)^{\mathrm{T}},\\
  \hat B_{\mathrm{m}}^3(1;0;0,0;0,\ldots)^{\mathrm{T}}&=(0;0;0,0;1,1,2,1,1;0,\ldots)^{\mathrm{T}},
\end{align*}
and so forth, as expected.

Let $\Vb$ denote the vector space of exponential tree series \eqref{eq:exptreeseries} represented by vectors of the form~\eqref{eq:expvector}. 
We can now give a second linearised flow prescription corresponding to that in Prescription~\ref{prescription:linearisedflow},
in the context of $\Vb$, as follows. 
\begin{prescription}[Linearised flow: reprise]\label{prescription:Grassmannian}
  Suppose that the linear operators $\mathcal P, \mathcal Q, \mathcal G\colon\Vb\to\Vb$ satisfy $\mathcal P(0)=\mathcal Q(0)=\id$,
  and the linear system of equations,
  \begin{equation*}
    \pa_t\mathcal P=O,\quad \pa_t\mathcal Q=-\hat B_{\mathrm{m}}\,\mathcal P\quad\text{and}\quad\mathcal P=\mathcal G\,\mathcal Q.
  \end{equation*}
\end{prescription}
For this linearised flow, the first two equations imply,
\begin{equation*}
\mathcal P=\id \qquad\text{and}\qquad \mathcal Q=\id-t\hat B_{\mathrm{m}}.
\end{equation*}
Then for sufficiently small $t$ we have $\mathcal G=(\id-t\hat B_{\mathrm{m}})^{-1}$.
Relating this linearised flow to the linearised flow Prescription~\ref{prescription:linearisedflow}, we naturally observe that
$\pf=\mathcal P({\tiny \begin{forest} for tree={grow'=90, parent anchor=center,child anchor=center} [] \path[fill=black] (.anchor) circle[radius=1.5pt]; \end{forest}})$
and the linearised flow solution $\gf$ is given by
$\gf=\mathcal G({\tiny \begin{forest} for tree={grow'=90, parent anchor=center,child anchor=center} [] \path[fill=black] (.anchor) circle[radius=1.5pt]; \end{forest}})$,
where recall `{\tiny \begin{forest} for tree={grow'=90, parent anchor=center,child anchor=center} [] \path[fill=black] (.anchor) circle[radius=1.5pt]; \end{forest}}'
represents the data $\xi({\tiny \begin{forest} for tree={grow'=90, parent anchor=center,child anchor=center} [] \path[fill=black] (.anchor) circle[radius=1.5pt]; \end{forest}})\in\Hb$.
We can also see this from another perspective. The solution \eqref{eq:mainsoln} for $\gf$ is an exponential series of the form \eqref{eq:exptreeseries} with $\hat\hf_\tau=\omega(\tau)\,t^{|\tau|}$.
Using the properties of the weight character $\omega$ from Definition~\ref{def:weightchar}, it is straightforward to show that the corresponding vector $\hat\ff$ of the form \eqref{eq:expvector},
satisfies $\pa_t\hat\ff=B_{\mathrm{m}}\hat\ff$. Note here, $B_{\mathrm{m}}$ is the non-graded form of $\hat B_{\mathrm{m}}$. It is the same as $\hat B_{\mathrm{m}}$ but a factor
corresponding to the grade of the tree at that position in the vector is attached first, before applying $\hat B_{\mathrm{m}}$.

Second, we now explore new numerical simulation methods that can be constructed by approximating the general solution form~\eqref{eq:mainsoln} in Theorem~\ref{thm:main},
for general frequency kernels, as well as in particular, separable frequency kernels.
Consider the exponential tree series solution~\eqref{eq:mainsoln}, which is given in $\R\la\mathbb T\ra$, transformed
to the form~\eqref{eq:solnseriesapp} in $\Hb(\xi,\star)$, the algebra of convergent exponential tree series in the sense of condition~\eqref{eq:conv}.
Suppose we truncate~\eqref{eq:solnseriesapp} to only include trees of grade, up to and including, $N\in\mathbb N$. Denote this
truncated series by,
\begin{equation}\label{eq:approx}
\overline{\gf}_N=\sum_{|\tau|\leqslant N} \frac{t^{|\tau|}}{|\tau|!}\omega(\tau)\,\xi(\tau).
\end{equation}
Let $\mathbb S_N(\xi)\coloneqq\{\xi(\tau)\colon \tau\in\mathbb T~\text{and}~|\tau|\leqslant N\}$ denote the
set of tree-parametrised terms $\xi(\tau)$ in the truncated expansion~\eqref{eq:approx}.
We construct a new, simple numerical method, based on the truncated expansion~\eqref{eq:approx},
to integrate Smoluchowski's coagulation equation for a general frequency kernel $K$, over a global time interval $[0,T]$ for some $T>0$.
Herein, we suppose $T$ precedes any gelation time. The basic aspects we need to consider are:\smallskip

1. \emph{Time discretisation:} If $T$ is very small, we may not need to discretise at all. We might be able to simply evaluate
the set $\mathbb S_N(\xi)$ for the initial data $\xi$ only, for sufficiently large $N$. Then the solution is given by $\overline{\gf}_N$,
or in coagulation variables as ${\mathcal H}^{-1}\overline{\gf}_N$. However, in general, we need to discretise the interval $[0,T]$ into, say,
$M$ equal subintervals (for simplicity for here) so that $[0,T]=\cup_{m=0}^{M-1}[t_m,t_{m+1}]$, with $t_0=0$ and $t_M=T$.
Let $\rd t\coloneqq T/M$ denote the uniform time step.
We need to compute the set $\mathbb S_N$ anew at each time step $t_m$, $m\in\{1,\ldots,M-1\}$---in addition to computing it at the initial time $t_0=0$.
Since the computational effort for $\mathbb S_N$ is large for large $N$, there is a natural trade-off between the sizes of $M$ and $N$.
Determining the optimal trade-off is very much of interest.\smallskip

2. \emph{Evaluation of the expansion terms $\mathbb S_N$:} There are two components of this computational effort:\smallskip

(a)  \emph{Non-planar trees:} The `$\star$' product associated with Smoluchowski's equation~\eqref{eq:Smoluchowski} for the scalar field $g=g(x;t)$
is commutative and the exponential tree series solution~\eqref{eq:mainsoln} simplifies to an expansion over rooted non-planar binary trees.
See Remark~\ref{rmk:nonplanartrees} for the full details. This means that for example, the integrals represented by $\xi\star(\xi\star\xi)$ and $(\xi\star\xi)\star\xi$
are the same, and can be combined. At grade $3$, there are only two distinct non-planar binary trees. Thus for the commutative case, which includes
the original Smoluchowski equation for a general frequency kernel $K$, there is a very significant reduction in computation effort associated
with computing the set $\mathbb S_N$ at each time step $t_m$, $m\in\{0,\ldots,M-1\}$.\smallskip

(b)\emph{Generalised transform:} We opted for the Bernstein transform in Section~\ref{sec:coagulation}, but in particular circumstances,
another linear transformation---such as the modified Bernstein cosine transform---might significantly improve the effort associated with computing $\mathbb S_N$.
Indeed, in practice, the fast Fourier transform works extremely well.\smallskip

With regards the time discretisation in item 1, the local error associated with the truncated approximation~\eqref{eq:approx} which includes trees of grade $N$,
across any time interval $[t_{m},t_{m+1}]$, is $\mathcal O\bigl((\rd t)^{N+1}\bigr)$. Since this local error accumulates linearly at leading order over the global time
interval $[0,T]$, the global time discretisation error is $\mathcal O\bigl((\rd t)^N\bigr)$. This of course assumes we have evaluated the $\mathbb S_N$ sufficiently
accurately at each time step. With regards item~2, let us consider part (a) first; part (b) is implicitly realised further below.
Recall Remark~\ref{rmk:nonplanartrees} in Section~\ref{sec:coagulation}.
In Tables~\ref{table:treesupto3} and \ref{table:treesof4}, we list all the planar binary trees $\tau$ up to, and including, those of grade $4$, together with
their associated weights $\omega(\tau)$ and symmetries $\sigma(\tau)$. Recall that the symmetry $\sigma(\tau)$ associated with a non-planar tree determines
the number of non-planar copies, given by $2^{\sigma(\tau)}$, of itself it generates. Let us consider the triples $\bigl(\tau,\omega(\tau),\sigma(\tau)\bigr)$ 
of non-planar trees at each grade. Recall the word-coding representation we introduced in Procedure~\ref{proc:treeorder}. At grade $1$ there is
only one non-planar tree `\raisebox{-1pt}{{\tiny \begin{forest} for tree={grow'=90, parent anchor=center,child anchor=center, l=0cm,inner ysep=0cm,edge+=thick} [[][]] \end{forest}}}'
equivalent to the triple $(1,1,0)$; see Table~\ref{table:treesupto3}. The triple of the non-planar tree of grade $2$ is $(12,1,1)$, while
at grade $3$ there are only two non-planar trees with triples $(123,1,2)$ and $(212,2,0)$.
Note, for non-planar trees, we choose the \emph{convention} to use the planar tree with the lowest word-order coding as the non-planar representative tree.
The triples for all the non-planar trees of grades $4$, $5$ and $6$ are given in Table~\ref{table:nonplanartrees}.
In the Table, we can compute the weight associated with each tree from weights of lower grade trees, using the formula in Definition~\ref{def:weightchar}.
The symmetries $\sigma(\tau)$ are easily computed in the manner outlined in Remark~\ref{rmk:nonplanartrees}. Further, we note that
$\xi(\raisebox{-1pt}{{\tiny \begin{forest} for tree={grow'=90, parent anchor=center,child anchor=center, l=0cm,inner ysep=0cm,edge+=thick} [[][]] \end{forest}}})=\xi(1)$
is given by $\xi(1)=\xi(0)\star\xi(0)$,
where $\xi(0)=\xi({\tiny \begin{forest} for tree={grow'=90, parent anchor=center,child anchor=center} [] \path[fill=black] (.anchor) circle[radius=1.5pt]; \end{forest}})$,
or more abstractly $1=0\star0$. This is naturally just grafting. Thus we also have $\xi(12)=\xi\star\xi(1)$, or abstractly $12=0\star1$, and also we have, abstractly,
$123=0\star12$ and $212=1\star1$. All the non-planar binary trees up to, and including, grade $3$, are thus generated from lower grade non-planar trees in this way.
In Table~\ref{table:nonplanartrees} in the final column, we indicate how all the non-planar binary trees of grades $4$, $5$ and $6$ can be similarly generated.
Let $\Sb_N^{\mathrm{np}}$ denote the subset of $\Sb_N$ consisting of non-planar binary trees, using the representation convention outlined above.
The truncated binary tree expansion for non-planar binary trees has the form,
\begin{equation}\label{eq:approxnp}
\overline{\gf}_N=\sum_{\tau\in\Sb_N^{\mathrm{np}}} \frac{t^{|\tau|}}{|\tau|!}\omega(\tau)\,2^{\sigma(\tau)}\,\xi(\tau).
\end{equation}

\begin{table}
  \caption{We list all the non-planar binary trees of grades $4$, $5$ and $6$. For each word-coded tree in the left column,
    we give its weight $\omega(\tau)$ and symmetry $\sigma(\tau)$ in the second and third columns, and how it is
    generated from lower grade non-planar trees in the final column.}
\label{table:nonplanartrees}
\begin{center}
  \setlength{\baselineskip}{2\baselineskip}
\begin{tabular}{|c|c|c|c|}
\hline
$\phantom{\biggl|}$ tree $\phantom{\biggl|}$ & weight & symm.\/ & generation \\ \hline
1234 & 1 & 3 & $0\star123$\\
1323 & 2 & 1 & $0\star212$\\
2123 & 3 & 2 & $1\star12$\\\hline
12345 & 1 & 4 & $0\star1234$\\
12434 & 2 & 2 & $0\star1323$\\
13234 & 3 & 3 & $0\star2123$\\
21234 & 4 & 3 & $1\star123$\\
21323 & 8 & 1 & $1\star212$\\
23123 & 6 & 2 & $12\star12$ \\\hline
123456 & 1  & 5 & $0\star12345$\\
123545 & 2  & 3 & $0\star12434$\\
124345 & 3  & 4 & $0\star13234$\\
132345 & 4  & 4 & $0\star21234$\\
132434 & 8  & 2 & $0\star21323$\\
134234 & 6  & 3 & $0\star23123$\\
212345 & 5  & 4 & $1\star1234$\\
212434 & 10 & 2 & $1\star1323$\\
213234 & 15 & 3 & $1\star2123$\\
231234 & 10 & 4 & $12\star123$\\
231323 & 20 & 2 & $12\star212$ \\\hline
\end{tabular}
\end{center}
\end{table}

\begin{remark}[Algebra of word-codes] 
  In general, for \emph{planar} binary trees, we can define an algebra isomorphic to $\mathbb R\la\mathbb T\ra$, based on the
  word-codes in Procedure~\ref{proc:treeorder}. The product `$\star$' is this algebra would be $a_1\cdots a_n\star b_1\cdots b_m=(a_1+1)\cdots(a_{n}+1)1(b_1+1)\cdots(b_m+1)$,
  with the first simple cases taking the form $0\star0=1$, $0\star 1=12$, $1\star 0=21$, $0\star12=123$, $0\star21=132$, $1\star1=212$, and so forth.
\end{remark}

Let us now outline a practical approach to evaluating the set $\mathbb S_N^{\mathrm{np}}$ at each time step $t_m$, $m\in\{0,1,\ldots,M-1\}$. 
Given the details just above, we essentially need to compute,
\begin{equation}\label{eq:prodcompute}
 \bigl(\xi\star\eta\bigr)(s)=\int_{[0,\infty)^2}\!\! H(s,y,z)K(y,z)\,g(y)f(z)\,\rd y\,\rd z,
\end{equation}
for suitable pairs of functions $\xi$ and $\eta$. In this product, $\xi=\mathcal H g$ and $\eta=\mathcal H f$,
where $\mathcal H$ is the generalised transform with kernel $h$, as outlined in the Introduction and discussed in
Section~\ref{sec:coagulation}, and $H(s,y,z)\coloneqq\frac12(h(s,y+z)-h(s,y)-h(s,z))$. Let us immediately
restrict ourselves to the case when the frequency kernel is separable, in the the sense that,
\begin{equation}\label{eq:sepkernel}
 K(y,z)=k(y)k(z),
\end{equation}
where $k$ is a given function $k\colon[0,\infty)\to[0,\infty)$. One reason for this choice is to compare
our results to those in Filbert and Lauren\c cot~\cite{FL}.
We outline at the end of this section our strategy for more general frequency kernels.
The Bernstein transform we considered in Section~\ref{sec:coagulation} is a natural choice for the separable 
frequency kernel case as $H(s,y,z)=\frac12(1-\mathrm{e}^{-sy})(1-\mathrm{e}^{-sz})$, and in this case the
double integral in \eqref{eq:prodcompute} decouples to become the real product of two integrals.
We wish to retain this property and take advantage of the computational efficiency of the fast Fourier transform.
For convenience, we define the simple linear function $\ell\colon[0,\infty)\to[0,\infty)$ to be,
\begin{equation*}
\ell\colon x\mapsto x.
\end{equation*}    
Given the domain of integration in \eqref{eq:prodcompute} involves the positive quadrant, a natural choice 
is to suppose $h$ generates the modified Bernstein cosine transform, i.e.\/ we take $h(x,s)=2x\,\cos(2\pi sx)$.
With this choice,
\begin{align*}
  H(s,y,z)=&\;y\cos(2\pi sy)\,\bigl(\cos(2\pi sz)-1\bigr)\\
  &\;+\bigl(\cos(2\pi sy)-1\bigr)\,z\cos(2\pi sz)\\
  &\;-y\sin(2\pi sy)\,\sin(2\pi sz)\\
  &\;-\sin(2\pi sy)\,z\sin(2\pi sz).
\end{align*}
If we subsitute this form for $H$ and the separable frequency kernel form~\eqref{eq:sepkernel}
into the product~\eqref{eq:prodcompute}, it decomposes into the sum of four terms, each one
of which is the real product of two factors. The factors are either cosine, Bernstein cosine, or sine transforms
of $kg$, $kf$, $\ell kg$ or $\ell kf$. These arguments can be extended as even or odd functions on the whole
real line and thus the transforms in the factors can be expressed in terms of Fourier transforms of the extended arguments.
Hence we can in principle use the fast Fourier transform to compute all the factors described.
Note, for example, for the first term in binary tree expansion \eqref{eq:approxnp},
given initial data such as $g_0=\exp(-x)/x$, then in the first time step the expression $\xi\star\xi$ can be computed analytically.

In practice, using the fast Fourier transform to evaluate \eqref{eq:prodcompute} for separable frequency kernels
is even more straightforward. We use the following form for the Fourier transform $\mathfrak f$ of the function $f$
and its inverse, 
\begin{align*}
  \mathfrak f(s)&\coloneqq\int_{\R}\exp(2\pi\mathrm{i}sx)\,f(x)\,\rd x,\\
            f(x)&\coloneqq\int_{\R}\exp(-2\pi\mathrm{i}sx)\,\mathfrak f(s)\,\rd s.
\end{align*}
Let us also outline the discrete Fourier transform (DFT). Understanding how this works is crucial to the numerical simulation method we propose.
See Press \textit{et al.} \cite{PTVF} for more details. Suppose we are given a function $f$ on a finite
set of equispaced nodes $x_n\coloneqq \nu h$, $\nu=0,1,\ldots,n-1$. Here, $h$ is the discrete `spatial' scale.
Set $f_\nu\coloneqq f(x_\nu)$. Then the Fourier transform $\mathfrak f$ is approximated by,
\begin{equation}\label{eq:DFT}
\mathfrak f_k\approx h\sum_{\nu=0}^{n-1}\exp(2\pi\mathrm{i}\nu\,k/n)\,f_\nu,
\end{equation}
where $s_k=k/nh$, $k=-n/2,\ldots,n/2-1$, and $\mathfrak f_k\coloneqq\mathfrak f(s_k)$.
The fast Fourier transform (FFT) computes the sum on the right, dropping the $h$ prefactor.
The discrete inverse Fourier transform is given by,
\begin{equation}\label{eq:iDFT}
f_\nu=\frac{1}{n}\sum_{k=-n/2}^{n/2-1}\exp(-2\pi\mathrm{i}\nu\,k/n)\,\mathfrak f_k.
\end{equation}
For the FFT, the set of $n$ frequencies $k$ are reordered. For convenience we denote the Fourier transform of $f$ by $\mathcal F(f)$. 
Now suppose $h(x,s)\coloneqq x\exp(2\pi\mathrm{i}sx)$. Then in this case,
\begin{align}
  H(s,y,z)=&\;\tfrac12\Bigl(y\exp(2\pi\mathrm{i} sy)\,\bigl(\exp(2\pi\mathrm{i} sz)-1\bigr)\nonumber\\
  &\;+\bigl(\exp(2\pi\mathrm{i} sy)-1\bigr)\,z\exp(2\pi\mathrm{i} sz)\Bigr).\label{eq:Hform}
\end{align}
If we substitute this form for $H$ into \eqref{eq:prodcompute}, using the separable frequency kernel form~\eqref{eq:sepkernel},
we find that,
\begin{align}
  \xi\star\eta=&\;\tfrac12\Bigl(\mathcal F(\ell kg)\,\bigl(\mathcal F(kf)-\mathcal F_0(kf)\bigr)\nonumber\\
  &\;+\bigl(\mathcal F(kg)-\mathcal F_0(kg)\bigr)\,\mathcal F(\ell kf)\Bigr), \label{eq:discreteprod}
\end{align}
where $\mathcal F_0(f)$ denotes the Fourier transform evaluated at $s=0$. This result is not
correct unless we extend the arguments $kg$, $\ell kg$, $kf$ and $\ell kf$ as even or odd functions
on the whole real line. Given one of its original applications, signal processing, the DFT/FFT
is ideally suited to the situation we now find ourselves in.
In practice we are given initial data $g_0=g_0(x)$ for $x\in[0,\infty)$, and it is natural 
to sample this at the nodes $x_\nu\coloneqq \nu h$ for $\nu=0,1,\ldots,n-1$, where $h\coloneqq L/n$.
Here $[0,L]$, for $L>0$, is a sufficiently large truncation of the coagulation domain $[0,\infty)$.
Often we are given data that is singular at the origin, for example $g_0(x)=\exp(-x)/x$, and
in this instance we take $x_\nu=(\nu+1) h$, for $\nu=0,1,\ldots,n-1$. Indeed this is our
modus operandi herein. However, note that the DFT in \eqref{eq:DFT}, naturally  does not reference the 
actual nodal positions $x_\nu$, but utilises the $n$ function nodal values $f_\nu$ and that the nodes are equispaced.
It generates a discrete transform at the frequencies $k$, indicated above. When computing the inverse DFT via \eqref{eq:iDFT},
we only use the $n$ transform values $\mathfrak f_k$.
We are thus free to chose the equispaced nodes we wish to evaluate our original function at, thus determining its locale on the real line.
Naturally we choose the nodal points so the locale is in $[0,\infty)$, in particular we choose $x_\nu=(\nu+1) h$, for $\nu=0,1,\ldots,n-1$.
Thus \emph{in practice}, having sampled the functions $g=g(x)$ and $f=f(x)$ at the nodal points $x_\nu$, as well as $k$ and $\ell$,
we can replace the Fourier transforms $\mathcal F$ of the arguments $kg$, $\ell kg$, $kf$ and $\ell kf$ in \eqref{eq:discreteprod} by the DFT/FFT.

It remains to describe the overall algorithm, implement it for some poignant examples, and demonstrate how our method can be improved
and generalised to wider classes of frequency kernels $K$. To describe the overall algorithm, it suffices to show how the approximate solution 
is evaluated over the first time step, as subsequent time steps simply repeat the process. Let us consider the case when we use the
truncated non-planar binary tree expansion \eqref{eq:approxnp} for $N=3$, which generates a third order integrator in time. Given initial
data $g_0$, we thus need to compute,
\begin{align}
  \overline{\gf}_3=&\;\xi+(\rd t)\,\xi(1)+\frac{1}{2!}(\rd t)^2\,\xi(12)\nonumber\\
  &\;+\frac{1}{3!}(\rd t)^3\,\bigl(4\cdot\xi(123)+2\cdot\xi(212)\bigr). \label{eq:order3}
\end{align}
To compute $\xi=\xi(0)$, we simply compute the FFT of $g_0$, having sampled $g_0$ on the $n$ nodes $x_\nu=(\nu+1) h$, for $\nu=0,1,\ldots,n-1$.
For the first term on the right above, we actually don't need to do this. At the end of each time step we evaluate the approxomate solution $g=g(x,t_m)$
in the coagulation space, as explained presently.
Then to compute $\xi(1)=\xi\star\xi$ we use \eqref{eq:discreteprod} with the FFT. We simply need to compute $h$ times the FFT of $kg_0$ and $\ell kg_0$.
The term $\mathcal F_0(kg_0)$ is $h$ times the first element in the FFT vector of $kg_0$. We then compute the inverse fast Fourier transform (iFFT)
of $\xi(1)$, call this $g(1)$, as we need this in the next step. To compute $\xi(12)=\xi\star\xi(1)$, we use \eqref{eq:discreteprod} with $g$ given by $g_0$,
and $f$ given by $g(1)$. We then compute the iFFT of $\xi(12)$, nominate this as $g(12)$. Finally, to compute $\xi(123)=\xi\star\xi(12)$,
we use \eqref{eq:discreteprod} with $g$ given by $g_0$, and $f$ given by $g(12)$. To compute $\xi(212)=\xi(1)\star\xi(1)$
we use \eqref{eq:discreteprod} with $g$ and $f$ both given by $g(1)$. We then compute the iFFT of $4\cdot\xi(123)+2\cdot\xi(212)$.
We can then evaluate $\overline{g}_3$, in coagulation space, by considering the linear combination of the corresponding terms in \eqref{eq:order3}.
The approximation $\overline{g}_3$ is then the initial data corrsponding to $g_0$ for the next time step.
Using Table~\ref{table:nonplanartrees} the procedure for computing higher time-order approximations is now straightforward.
For the third order case we considered here, the total number of FFTs required, including the iFFTs, is ten. And of course
the effort of each such FFT operation is proportional to $n\log n$. For the separable frequency kernel case, this compares favourably with the   
numerical method of Filbet and Lauren\c cot~\cite{FL}. 

We implement six numerical approximations based on the truncated non-planar binary tree expansion \eqref{eq:approxnp}, as just described,
for $N\in\{1,2,3,4,5,6\}$, for different numbers of time steps $M$ on a global time interval $[0,T]$. In all cases the truncated `spatial' domain
is $[0,L]$ with $L=100$. We used $n=2^{20}$ modes to evaulate the required FFTs in all cases, except where stated otherwise. As in some
of the examples outlined in Filbet and Lauren\c cot~\cite{FL}, we set,
\begin{equation*}
k(x)\coloneqq x^{\lambda/2}.
\end{equation*}
We consider three separate cases, when the parameter $\lambda$ is set to be $\lambda=2$, $\lambda=3/2$ and $\lambda=2/3$.
The first case, $\lambda=2$, corresponds to the solvable multiplicative frequency kernel case. Indeed, for the initial data 
$g_0=\exp(-x)/x$ there is a well-known closed form solution given by,
\begin{equation}\label{eq:Bessel}
g(x;t)=\exp\bigl(-(1+t)x\bigr)\,I_1(2x\sqrt{t})/x^2\sqrt{t},
\end{equation}
for $t\leqslant1$, where $I_1$ is the modified Bessel function of the first kind. 
This solution extends beyond the gelation time $t=1$; see (ii) in Section~\ref{sec:discussion}.
In the top panel in Figure~\ref{fig:errors}, we give a log-log plot of global error versus 
the number of steps $M$ used to compute the approximation at time $T=0.5$, for all six cases $N\in\{1,2,3,4,5,6\}$.
The $L^2$ error was computed by comparing the approximations to the exact solution~\eqref{eq:Bessel}.
All the methods have the order anticipated. We used the same number of nodes $n=2^{20}$ for all the methods.
We observe that the higher order methods error curves flatten off around $10^{-3}$, this just an artifact
of the `spatial' discretisation error, determined by $n$, starts to exceed the time step error.

The second case, $\lambda=3/2$, also corresponds to a gelation case. For the same initial data $g_0(x)=\exp(-x)/x$,
we computed the solution up to time $t=0.9$ for all six integrators. In the middle panel in Figure~\ref{fig:errors},
we give a log-log plot of global error versus the number of steps $M$. 
The $L^2$ error was computed by comparing the approximations computed with $n=2^{20}$ modes versus the
sixth order approximation computed with the smallest step size and $n=2^{21}$ `spatial' modes.
This log-log error plot looks very much like that for the $\lambda=2$ case, and again the
flattenning off of the error curves for the higher order methods is just an artifact of the
`spatial' discretisation error exceeding the time step error. We give a log-log plot of
the solution in the left panel in Figure~\ref{fig:solution}.

The third case, $\lambda=2/3$, is non-gelling. For the initial data $g_0(x)=\exp(-x)$,
in the bottom panel in Figure~\ref{fig:errors}, we give a log-log plot of global error versus the number of steps $M$, computed to the time $T=1.5$.
The $L^2$ error was computed using the same procedure we outlined for the case $\lambda=3/2$ just above.
Again the log-log error plot has the same characteristics as the previous two cases above.
In the right panel in Figure~\ref{fig:solution}, we give a log-log plot of the solution computed to
the time $T=15$ in this case---using $M=2^7$ time steps. This compares well with the corresponding
plot in Filbet and Lauren\c cot~\cite{FL}, namely in Figure~11, in the lower left panel therein.
Overall, our numerical method based on the truncated non-planar binary tree expansion \eqref{eq:approxnp} and FFT computations
appears robust and efficient.

\begin{figure}
  \begin{center}
  \includegraphics[width=9.5cm,height=7cm]{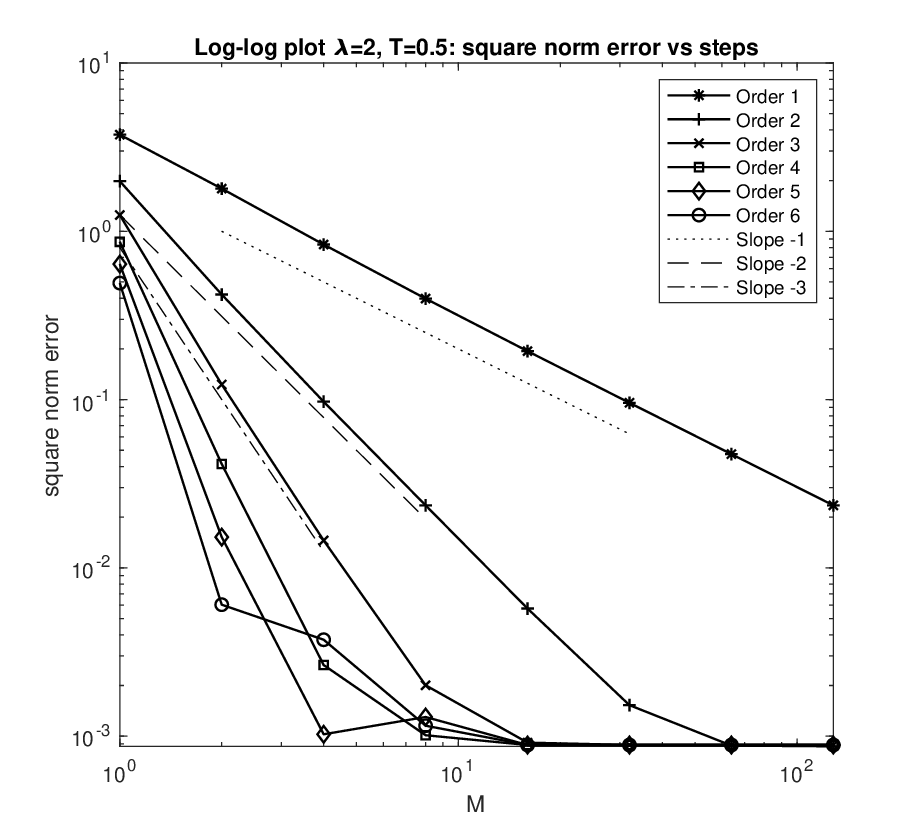}\\
  \includegraphics[width=9.5cm,height=7cm]{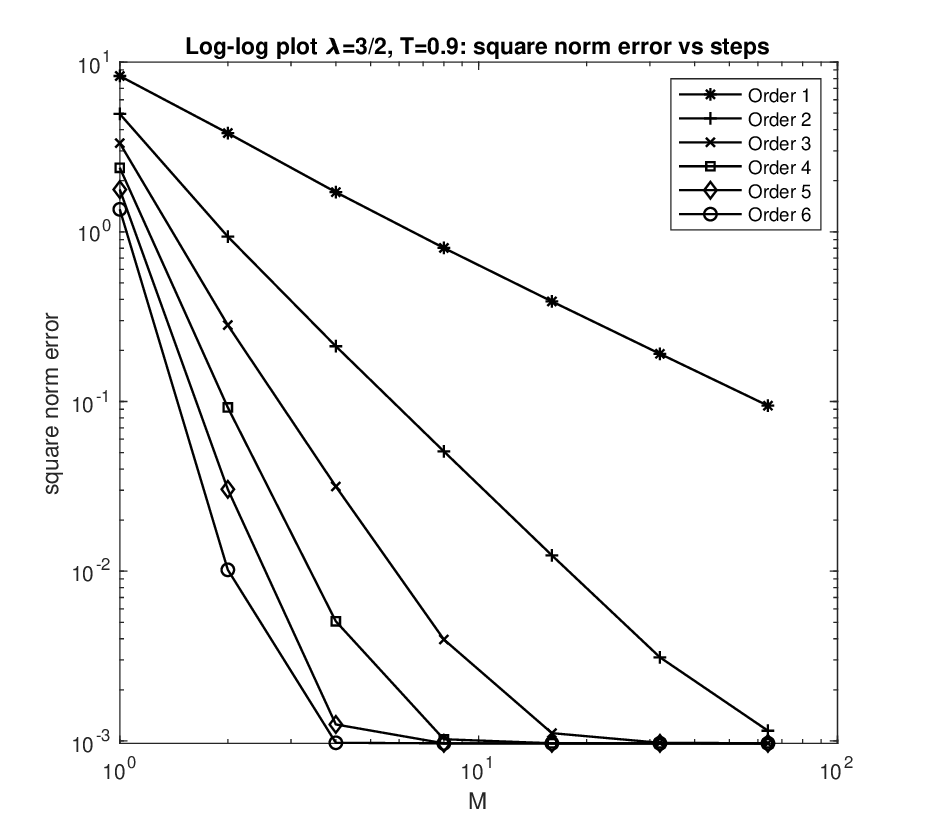}\\
   \includegraphics[width=9.5cm,height=7cm]{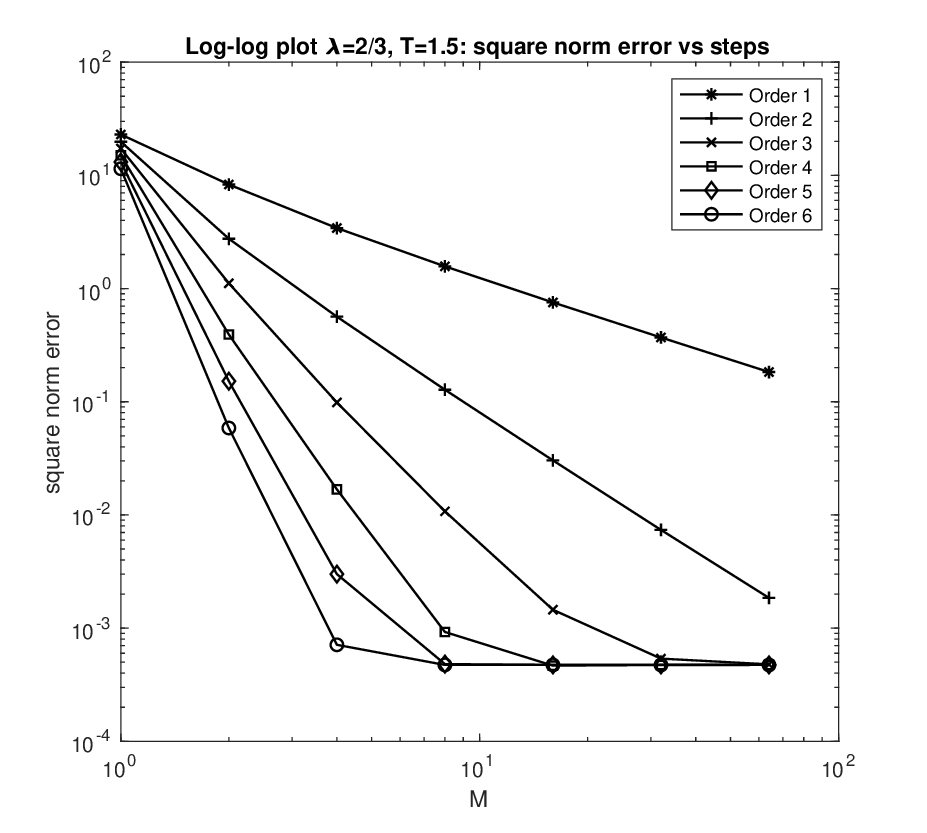}
\end{center}
  \caption{Log-log plots of the mean square error versus the number of time steps are shown in the
    top panel ($\lambda=2$, $T=0.5$), the middle panel ($\lambda=3/2$, $T=0.9$) and the bottom panel ($\lambda=2/3$, $T=1.5$).}
\label{fig:errors}
\end{figure}

\begin{figure}
  \begin{center}
  \includegraphics[width=9.5cm,height=7cm]{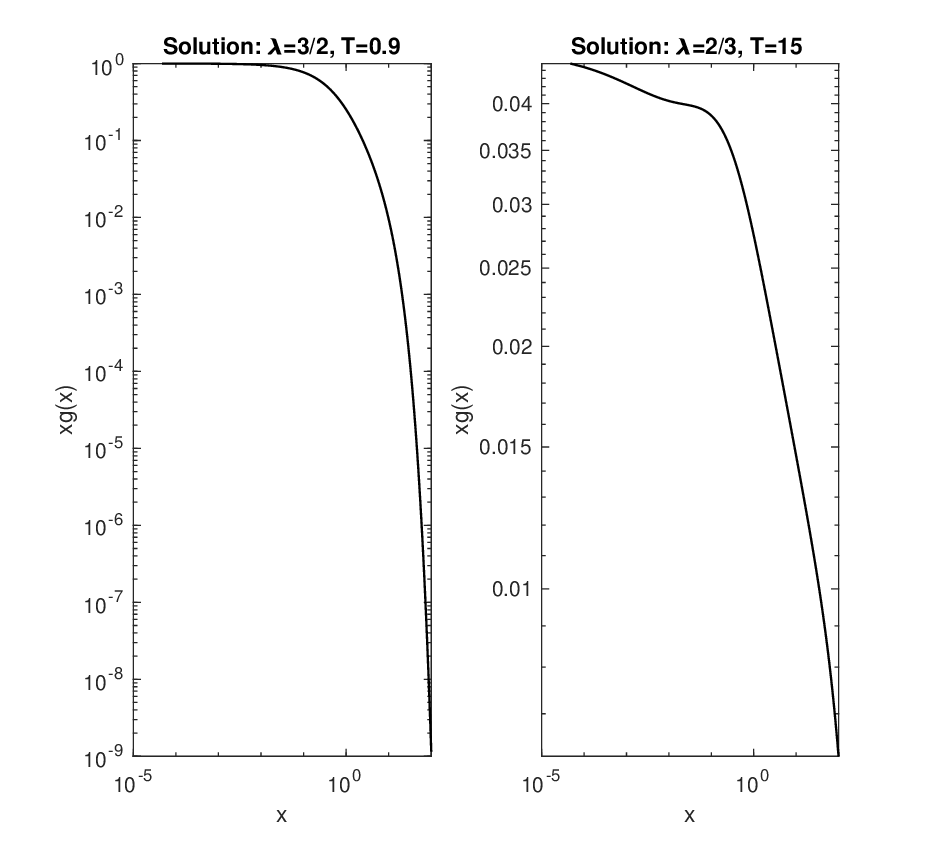} 
\end{center}
  \caption{The approximate solution for the cases $\lambda=3/2$, $T=0.9$ (left) and $\lambda=2/3$, $T=15$ (right) are shown,
    both computed using $2^{21}$ modes with $2^7$ timesteps.}
\label{fig:solution}
\end{figure}

The method we have outlined can be improved and generalised in several ways as follows.\smallskip

1. \emph{Non-uniform fast Fourier transform:} The FFT computations above relied on the nodal points $x_\nu$ being equispaced
across the domain $[0,L]$, where $L=100$. However, for all the cases we considered, for large $x$, the initial data decays like `$\exp(-x)$'.
It would therefore seem sensible to distribute the nodal points to this `distribution', so that more nodal points are concentrated around smaller values of $x$. 
The non-uniform fast Fourier transform (NUFFT) does indeed exist and the effort required is still proportional to $n\log n$,
though with a larger constant. In its implementation in our context here, for an exponential distributed set of nodes $x_\nu$, we would compute the NUFFT
on the same set of frequencies as above. For a different initial data, we may wish to choose a differently distributed set of nodes $x_\nu$.
The inverse NUFFT can also be computed, though this requires some additional steps. For more details, see Kircheis and Potts~\cite{KP}.
In particular, once implemented, we could in principle, evaluate the solution for very small nodal points, eg. $10^{-40}$,
as in Filbet and Lauren\c cot~\cite{FL}.\smallskip

2. \emph{General additive kernels $K(x,y)=k(x)+k(y)$:} The method we have outlined is straightforwardly adapted to this class of
frequency kernels, which were also considered by Filbet and Lauren\c cot~\cite{FL}.\smallskip

3. \emph{General kernels:} The method we have outlined is easily adapted to any frequency kernels $K=K(y,z)$ consisting of a finite linear combination
of the forms $k_i(y)k_j(z)$, for $i,j=1,\ldots,l$, and such that $k_i(z)k_j(y)=k_i(y)k_j(z)$.
However, when $l$ is large, or if $K$ cannot be expressed as such a finite linear combination, then we might be able to proceed as follows.
We set $h(x,s)=x\exp(2\pi\mathrm{i} sx)$ as above. Then given $g$ and $f$, computing the product \eqref{eq:prodcompute},  
with $H$ given by \eqref{eq:Hform}, amounts to computing the two-dimensional Fourier transform, or two-dimensional FFT, along the diagonal set of wavenumbers.
In principle, we require the one-dimensional iFFT of the resulting expression for $\xi\star\eta$, in order to implement our method, as outlined
after \eqref{eq:order3}.

\section{Discussion}\label{sec:discussion}
There are many connections and extensions to the material herein. A brief summary of examples is as follows.

(i) \emph{Grassmannian flow:} In Doikou \text{et al.\/} \cite{DMSW-coagulation} we outlined in detail how the constant, additive and multiplicative frequency
kernels can be considered as Grassmannian flows. In principle, we can extend these cases to the general frequency kernel case, i.e.\/
the solution flow~\eqref{eq:mainsoln}--\eqref{eq:matchsoln} can formally be formulated as a Grassmannian flow as follows.
Recall the vector space $\Vb$ from Section~\ref{sec:simulation}. 
We now construct the Grassmannian $\mathrm{Gr}(\Vb\oplus\Vb^\perp;\Vb)$, the set of all subspaces $\mathbb W$ of $\Vb\oplus\Vb^\perp$ such that:
(1) the orthogonal projection $\mathrm{pr}\colon\mathbb W\to\Vb$ is a Fredholm operator, indeed a Hilbert--Schmidt perturbation of the identity;
(2) the orthogonal projection $\mathrm{pr}\colon\mathbb W\to\Vb^\perp$ is a Hilbert--Schmidt operator. See Pressley and Segal~\cite{PS}.
Herein we assume $\Vb^\perp$ is isomorphic and isometric to $\Vb$. Any such subspace $\mathbb W$ has a representation of the form, $\mathcal W=(\mathcal Q; \mathcal P)$,
where $\mathcal Q$ is a Fredholm operator on $\Vb$ which is a Hilbert--Schmidt perturbation of the identity,
and $\mathcal P$ is a Hilbert--Schmidt operator on $\Vb$. Let $\Vb_0$ denote the canonical subspace with the representation, $\mathcal V_0=(\id; O)$,
where $O$ is the infinite matrix of zeros. The projections $\mathrm{pr}\colon\mathbb W\to\Vb_0$ and $\mathrm{pr}\colon\mathbb W\to\Vb_0^\perp$ 
respectively give, $\mathcal W^\parallel\coloneqq(\mathcal Q; O)$ and $\mathcal W^\perp\coloneqq(O; \mathcal P)$.
These projections are possible provided $\mathrm{det}_2\,\mathcal Q\neq0$.
The subspace spanned by the columns of $\mathcal W^\parallel$ coincides with $\Vb_0$,
and indeed, $\mathcal Q^{-1}$ transforms $\mathrm{span}\{\mathcal W^\parallel\}$ to $\Vb_0$.
Under this transformation, the representation $\mathcal W$ for $\mathbb W$ becomes, $(\id; \mathcal G)$,
where $\mathcal G=\mathcal P\,\mathcal Q^{-1}$. Any subspace that can be projected onto $\Vb_0$ can be represented in this way and vice-versa.
The Hilbert--Schmidt operators $\mathcal G$ parametrise all subspaces $\mathbb W$ that can be projected in this way. If $\mathrm{det}_2\,\mathcal Q=0$,
a different coordinate patch must be chosen. For more details, see Pressley and Segal~\cite{PS} or Doikou \textit{et al.\/} \cite{DMSW-integrable},
as well as Beck \textit{et al.\/} \cite{BDMS1,BDMS2}. 
A possible Grassmannian flow prescription for our flow herein, corresponding to the linearised flow Prescription~\ref{prescription:linearisedflow},
is the linearised flow Prescription~\eqref{prescription:Grassmannian} for the linear operators $\mathcal P$, $\mathcal Q$ and $\mathcal G$.
This approach needs further investigation. In particular, an important aspect of Grassmannian flows in a given coordinate patch, is that flow singularities 
correspond to a poor representative patch, where $\mathrm{det}_2\mathcal Q=0$, but the solution flow can still be represented and continued in a different coordinate patch.
This is relevant to the following item on gelation. 

(ii) \emph{Computing beyond gelation:} Quoting from van Roessel and Shirvani~\cite{vRS}: ``The phenomenon whereby conservation of mass breaks down in finite time is known as gelation and is
physically interpreted as being caused by the appearance of an infinite `gel' or `superparticle'...''. In turn, also see Ernst \textit{et al.} \cite{EZH}.
Herein, for frequency kernels for which gelation occurs, we have focused on computing the solution up to the time of gelation.
During this interval of time, the quantity $M_t\coloneqq\int_0^\infty x\,g(x;t)\,\rd x$ is conserved by the flow.
However, post-gelation, quoting from Normand and Zambotti~\cite{NZ},
the superparticles ``do not count in the computation of the mass so from the gelation time on, $M_t$ starts to decrease''.
Indeed the solution can be computed, see for example, Leyvraz and Tschudi~\cite{LT}, van Roessel and Shirvani~\cite{vRS} and Normand and Zambotti~\cite{NZ}.
Using the binary tree expansion to compute the solution $\gf$, analytically and/or numerically, post-gelation, is very much of interest.
Indeed a Grassmannian flow context would seem to be a natural one to encapsulate the bahaviour before, during and after such a phase transition,
via different coordinate patch representations.

(iii) \emph{Kingman coalescent and Galton--Watson processes:} Smoluchowski coagulation models are naturally represented in terms of
coalescent stochastic processes, in particular Kingman coalescent processes or Galton-Watson branching processes (reversed in time).
See, for example, Aldous~\cite{Aldous}, Iyer \textit{et al.\/} \cite{ILP}, Harris \textit{et al.\/} \cite{HJR} and Johnston \textit{et al.\/} \cite{JKR}.
Etheridge~\cite[Ch.~2]{Etheridge} demonstrates the following insightful connection for the solution of a classical
quadratically semilinear parabolic partial differential equation. The solution can be expanded deterministically 
by an iterative procedure akin to that we performed for the coagulation equation~\eqref{eq:abstract}. The terms in
the solution expansion are indexed by rooted planar binary trees. Etheridge then shows that, at a given time, the terms at each grade
are given by the expectation across branches of a given branching process---by analogy with McKean's~\cite{McKeanII} solution
of the Kolmogorov--Petrovski--Piskunov equation via branching Brownian motion. It is thus natural to try to establish such a connection between
each term of the rooted planar binary tree expansion~\eqref{eq:mainsoln} and the branches of the
underlying Galton--Watson process, reversed in time. 

(iv) \emph{Multiple mergers:} A natural extension of the Smoluchowski coagulation model is to coalescent phenomena
involving multiple mergers. In principle the solution in such cases can be similarly expanded as an exponential trees series
analogous to~\eqref{eq:mainsoln} indexed by planar trees, except now the set of indexing planar trees would be more general
and could include the complete collection of $n$-ary planar trees reflecting the class of merger coalesence included.
See, for example, Iyer \textit{et al.\/} \cite{ILP} for more details on Smoluchowski models with multiple coalescence.
Also see Doikou \textit{et al.\/} \cite[Sec.~3]{DMSW-coagulation}.

(v) \emph{Decorated trees:} We considered a non-commutative, non-associative algebra with one generator and the non-associative product `$\star$'.
Naturally in general, we can construct such algebras with more than one generator. In this instance we can use the set of decorated binary
trees to represent the monomials in such an algebra---or in the multiple merger case, just decorated trees.
In the decorated binary tree context, we could extend the action of the budding operator $\mathcal B$ as follows.
For example, suppose there are three generators $\xi$, $\eta$ and $\zeta$ in $\Hb$, respectively represented by `$\circ$', `$\bullet$' and `$\diamond$' in $\mathbb T$. Then we could define,
\begin{equation*}
\mathcal B_\circ(\bullet)=\raisebox{-3pt}{{\tiny \begin{forest} for tree={grow'=90, parent anchor=center,child anchor=center, l=0cm,inner ysep=0pt,edge+=thick, s sep=0mm} [[$\circ$][$\bullet$]] \end{forest}}}
\qquad\text{and}\qquad 
\mathcal B_\circ\bigl(\raisebox{-3pt}{{\tiny \begin{forest} for tree={grow'=90, parent anchor=center,child anchor=center, l=0cm,inner ysep=0pt,edge+=thick, s sep=0mm} [[$\bullet$][$\diamond$]] \end{forest}}}\bigr)
=\raisebox{-4pt}{{\tiny \begin{forest} for tree={grow'=90, parent anchor=center,child anchor=center, l=0cm,inner ysep=0cm,edge+=thick} [[[$\circ$][$\bullet$]][$\diamond$]] \end{forest}}}
+\raisebox{-4pt}{{\tiny \begin{forest} for tree={grow'=90, parent anchor=center,child anchor=center, l=0cm,inner ysep=0cm,edge+=thick} [[$\bullet$][[$\circ$][$\diamond$]]] \end{forest}}}.
\end{equation*}
Also recall Remark~\ref{rmk:extend}, we could also consider extending this to include $\mathcal B_{\tau}$, for any $\tau\in\mathbb T$, which might
model multiple stage reactions depending on the tree attached. Note how $\mathcal B$ is left-budding---for each ``free'' branch we extend the existing bud to the right and attach
the new one to the left. Naturally there is a right-budding version as well.

(vi) \emph{Species:} The coagulation equation~\eqref{eq:abstract}, whose solution~\eqref{eq:mainsoln} is expressed as an exponential series in rooted planar binary
trees, can be expressed as an equation in the species of rooted planar binary trees. The extensions mentioned in (iii) above to more general coalescent processes are,
in principle, examples of further species of structures. See Bergeron \textit{et al.\/} \cite{Bergeron} for more details on species.

(vii) \emph{Multi-component coagulation:} There are multi-component generalisations of Smoluchowski's coagulation model \eqref{eq:Smoluchowski},
in particular in the context of atmospheric sciences, where ``clusters can be formed by different types of particles''; see Throm~\cite{Throm}.
Adapting our binary tree expansion approach to such generalisations appears to be straightforward, and again, very much of interest.

(viii) \emph{Hopf algebras of trees:} We mentioned in the introduction the depth and wide ranging applications of algebras of planar trees.
If the algebra of planar trees is endowed with the grafting product, one can define a compatabile co-product, and an antipode, and thus establish a Hopf algebra of such trees;
see Loday and Ronco~\cite{LR}. We mention, due to their more direct relevance here, the use of rooted planar trees in Lie group methods and backward error analysis in
Munthe--Kaas and Wright~\cite{M-KW} and Lundervold and Munthe--Kaas~\cite{LM-Kbackward,LM-K}.
Indeed, there is more than one Hopf algebra of planar trees, see Calaque \textit{et al.\/}~\cite{CE-FM}.
One co-product that would be useful in our analysis would be that which, for a given planar binary tree $\tau$,
it additively enumerates all the possible pairs of trees that when root grafted together, generate $\tau$.
Thus for example, for a given general tree series expansion with terms $\hat\ff_\tau\cdot\tau$, such a coproduct helps keep track of the origin of coefficients
that are generated by root grafting two such series together. This means that, in principle, there is an equivalent formulation of~\eqref{eq:mainsoln} 
as a co-product expansion. And, in principle, this could lead to a more abstract formulation of~\eqref{eq:mainsoln} in terms of Hopf algebra
endomorphisms, by analogy with such expansions in Ebrahimi--Fard \textit{et al.\/} \cite{E-FLMM-KW} and Ebrahimi--Fard \textit{et al.\/} \cite{E-FMPW}.
Lastly, we also mention here the work by Ishida~\cite{Ishida} on the Lie algebra of rooted planar trees,
Chapoton~\cite{Chapoton} on exponential-like series and Gerritzen~\cite{Gerritzen} on non-associative exponential series.

(ix) \emph{Free pre-Lie algebra:} Al Kaabi~\cite{AlKaabi} considers the free pre-Lie algebra structure associated with rooted planar trees.
The construction of numerical algorithms in the free pre-Lie algebra context is very much of interest.

(x) \emph{Branching Brownian motion and colloids:} Smoluchowski diffusion models incorporate the spatial Brownian motion
of clusters. They have applications in theory of colloids and sedimentation.
In such models, the density $g=g(x,\zeta;t)$ is recorded at position $\zeta$. A branching Brownian motion can be considered
as a ``Gaussian process indexed by the leaves of a Galton--Watson process''---Bovier~\cite{Bovier}. Viewed backwards in time
we observe a diffusive coalescent, and Smoluchowski diffusion models can be interpreted in this light; see Harris \textit{et al.\/} \cite{HJR}.
Can we extend the connection in (iii) above, here between the deterministic interative solution expansion and the underlying diffusive coalescent process,
in this case? See Etheridge~\cite{Etheridge} and Dynkin~\cite{Dynkin} for more details of the diffusive branching case, and
Berestycki and Berestycki~\cite{BB} and Berestycki \textit{et al.\/} \cite{Berestycki} for more details of the diffusive coalescent case.





\section{Declarations}

\subsection{Acknowledgement}
The author would like to thank the referees for their positive reports and suggestion to implement the numerical scheme.
This directly led to the material in the latter half of Section~\ref{sec:simulation} and helped to significantly improve the original manuscript.
The author is also very grateful to one of the referees for bringing \cite{LT} to his attention.

\subsection{Funding and conflicts or competing interests}
SJAM was supported by an EPSRC Mathematical Sciences Small Grant EP/X018784/1. 
There are no conflicts of interests or competing interests. 

\subsection{Data availability statement}
No data was used in this work. All the Matlab codes are provided in the electronic supplementary material.


\begin{thebibliography}{99}

\bibitem{AguiarSottile} Aguiar M, Sottile F 2006 Structure of the Loday--Ronco Hopf algebra of trees,
  \textit{Journal of Algebra} \textbf{295}(2), 473--511.

\bibitem{Aldous} Aldous DJ 1999
Deterministic and stochastic models for coalescence (aggregation and coagulation): a review of the mean-field theory for probabilists,
\textit{Bernoulli} \textbf{5}(1), 3--48.

\bibitem{AlKaabi} Al-Kaabi MJH 2014 Monomial bases for free pre-Lie algebras, \textit{S\'eminaire Lotharingien de Combinatoire} \textbf{71}, B71b.

\bibitem{Arcis} Arcis D, M\'arquez S 2022 Hopf algebras of planar trees and permutations, \textit{Journal of Algebra and its Applications}, 2250224.

\bibitem{BDMS1} Beck M, Doikou A, Malham SJA, Stylianidis I 2018
Grassmannian flows and applications to nonlinear partial differential equations,
\textit{Proc. Abel Symposium} 2018.

\bibitem{BDMS2} Beck M, Doikou A, Malham SJA, Stylianidis I 2018
Partial differential systems with non-local non-linearities: Generation and solutions, 
\textit{Phil. Trans. R. Soc. A} \textbf{376}, 2117, 195.

\bibitem{BB} Berestycki J, Berestycki N 2009 Kingman's coalescent and Brownian motion, arXiv:0904.1526

\bibitem{Berestycki} Berestycki J Harris SC, Kyprianou AE 2011 Travelling waves and homogeneous fragmentation,
  \textit{The Annals of Applied Probability} \textbf{21}(5), 1749--1794.

\bibitem{Bergeron} Bergeron F, Labelle G, Leroux P 2013 Introduction to the theory of species of structures, UQAM preprint.

\bibitem{Bertoin} Bertoin J 2002, Eternal solutions to Smoluchowski's coagulation equation with additive kernel and their probabilistic interpretations,
  \textit{The Annals of Applied Probability} \textbf{12}(2), 547--564.
  
\bibitem{Bovier} Bovier, A 2015 From spin glasses to branching Brownian motion --- and back? In Random Walks, Random Fields and Disordered Systems,
  Lecture Notes in Mathematics 2144, Eds. M. Biskup, J. \v{C}ern\'y and R. Koteck\'y, Springer.

\bibitem{Brouder} Brouder C 2000 Runge--Kutta methods and renormalization, \textit{Eur. Phys. J. C, Part. Fields} \textbf{12}(3), 521--534.
  
\bibitem{BR} Budiman RA, Ruda HE 2002
Smoluchowski ripening and random percolation in epitaxial Si$_{1-x}$Ge$_{x}$/Si(001) islands,
\textit{Phys. Rev. B} \textbf{65}, 045315. 

\bibitem{Butcher} Butcher JC 1972 An algebraic theory of integration methods, \textit{Math. Comput.} \textbf{26}(117), 79--104. 

\bibitem{Byrnes} Byrnes CI 1998 On the Riccati partial differential equation for nonlinear Bolza and Lagrange problems, 
\textit{Journal of Mathematical Systems, Estimation and Control} \textbf{8}(1), 1--54.

\bibitem{BJ} Byrnes CI, Jhemi A 1992 
Shock waves for Riccati partial differential equations arising in nonlinear optimal control, 
\textit{Systems, Models and Feedback: Theory and Applications (Capri 1992)}, 
Prog. Systems Control Theory 12. Boston, Birkhauser, 211--227

\bibitem{CE-FM} Calaque D, Ebrahimi--Fard K, Manchon D 2011 Two interactig Hopf algebras of trees: A Hopf algebraic approach to composition and substitution of B-series,
  \textit{Advances in Applied Mathematics} \textbf{47}(2), 282--308.

\bibitem{CI-ME} Carbonell F, Iturria--Medina Y, Evans AC 2018
Mathematical modelling of protein misfolding mechanisms in neurological diseases: a historical overview,
\textit{Frontiers in Neurology} \textbf{9}, 37. 

\bibitem{CdaC} Carr J, da Costa FP 1992 Instantaneous gelation in coagulation dynamics, \textit{Z. angew Math. Phys.} \textbf{43}, 974--983.

\bibitem{Chapoton} Chapoton F 2002 Rooted trees and an exponential-like series, arXiv:math/0209104.

\bibitem{CP} Chatel G, Pilaud V 2017 Cambrian Hopf algebras, \textit{Advances in Mathematics} \textbf{311}, 598--633.
  
\bibitem{Collet} Collet JF 2004 Some modelling issues in the theory of fragmentation-coagulation systems,
\textit{Comm. Math. Sci.} \textbf{1}, 35--54. 

\bibitem{CK} Connes A, Kreimer D 1998 Hopf algebras, renormalization and noncommutative geometry, \textit{Commun. Math. Phys.} \textbf{276}(2), 203--242.

\bibitem{CNDEBWBMGPM} Coraux J, N'Diaye AT, Engler M, Busse C, Wall D, Buckanie N, Meyer zu Heringdorf F-J,
van Gastel R, Poelsema B, Michely T 2009 Growth of graphene on Ir(111), \textit{New J. Phys.} \textbf{11}, 023006.

\bibitem{daC} da Costa FP 2015 Mathematical aspects of coagulation-fragmentation equations,
  in CIM Series in Mathematical Sciences 2, Mathematics of Energy and Climate Change,
  International Conference and Advanced School Planet Earth, Eds. J-P Bourguignon, R. Jeltsch, AA Pinto, M Viana, Springer, 83--162.
  
\bibitem{DT} Deaconu M, Tanr\'e E 2000
\textit{Smoluchowski's coagulation equation: probabilistic interpretation of solutions for constant, additive and multiplicative kernels},
Ann. Scuola Norm. Sup. Pisa Cl. Sci. (4) \textbf{XXIX}, 549--579.

\bibitem{DerridaRetaux} Derrida B, Retaux M 2014 The depinning transition in presence of disorder: a toy model,
\textit{J. Statist. Phys.\/} \textbf{156}, 268--290.

\bibitem{DMS} Doikou A, Malham SJA, Stylianidis I 2021
Grassmannian flows and applications to non-commutative non-local and local integrable systems,
\textit{Physica D} \textbf{415}, 132744.

\bibitem{DMSW-coagulation} Doikou A, Malham SJA, Stylianidis I, Wiese A 2023
Applications of Grassmannian flows to coagulation equations, \textit{Physica D}, \textbf{451}, 133771.

\bibitem{DMSW-integrable} Doikou A, Malham SJA, Stylianidis I, Wiese A 2021
Applications of Grassmannian flows to integrable systems, arXiv:1905.05035v2.

\bibitem{Dubovskii} Dubovskii PB 1994 \textit{Mathematical theory of coagulation}, Lecture Notes \textbf{23},
  Seoul: Global Analysis Research Center, Seoul National University.

\bibitem{Dynkin} Dynkin EB 2002 \textit{Diffusions, superdiffusions and partial differential equations},
AMS Colloquium Publications \textbf{50}.

\bibitem{E-FLMM-KW} Ebrahimi--Fard K, Lundervold A, Malham SJA, Munthe--Kaas H, Wiese A 2012
Algebraic structure of stochastic expansions and efficient simulation,
\textit{Proc. R. Soc. A} \textbf{468}, 2361--2382. (doi:10.1098/rspa.2012.0024)

\bibitem{E-FMPW} Ebrahimi--Fard K, Malham SJA, Patras F, Wiese A 2015
The exponential Lie series for continuous semimartingales, \textit{Proc. R. Soc. A} \textbf{471}.

\bibitem{EZH} Ernst MH, Ziff RM, Hendriks EM 1984 Coagulation processes with a phase transition, \textit{J. Colloid Interface Sci.} \textbf{97}(1), 266-277.

\bibitem{EMP} Escobedo M, Mischler S, Perthame B 2002 Gelation in coagulation and fragmentation models, \textit{Commun. Math. Phys.} \textbf{231}, 157--188.

\bibitem{ELMP} Escobedo M, Lauren\c{c}ot Ph, Mischler S, Perthame B 2003 Gelation and mass conservation in coagulation-fragmentation models,
  \textit{J. Differential Equations} \textbf{195}, 143--174.

\bibitem{Etheridge} Etheridge AM 2000 An introduction to superprocesses,
University Lecture Series (Providence R.I.)\textbf{20}, AMS.
  
\bibitem{Feller} Feller W, 1971 \textit{An introduction to probability theory and its applications},
Vol. 2, 2nd ed.,John Wiley \& Sons, Inc.

\bibitem{FG-BV} Figueroa H, Gracia-Bond\'{i}a JM, V\'arilly JC 2022 Fa\`a di Bruno Hopf algebras, arXiv:math/0508337v3

\bibitem{FL} Filbet F, Lauren\c cot P 2004 Numerical Simulation of the Smoluchowski coagulation equation, \textit{SIAM J. Sci. Comput.} \textbf{25}(6), 2004--2028.
  
\bibitem{Foissy} Foissy L 2013 An introduction to Hopf algebras of trees, Lecture Notes.
  
\bibitem{GLK} Galina H, Lechowicz JB, Kaczmarski K 2001
Kinetic models of the polymerisation of an AB$_2$ monomer,
\textit{Macromol. Th. Simul.} \textbf{10}(3), 174--178. 

\bibitem{GM} Gallay T, Mielke A 2003 Convergence results for a coarsening model using global linearization,
\textit{J. Nonlinear Sci.\/} \textbf{13}, 311-346.

\bibitem{Gerritzen} Gerritzen L 2004 Planar rooted trees and non-associative exponential series,
  \textit{Advances in Applied Mathematics} \textbf{33}, 342--365.

\bibitem{Gessel} Gessel IM 2016 Lagrange inversion, \textit{Journal of Combinatorial Theory, Series A} \textbf{144}, 212-249.

\bibitem{GL} Grossman R, Larson R 1989 Hopf-algebraic structure of families of trees, \textit{J. Algebra} \textbf{126}(1), 184--210.
  
\bibitem{GFK} Guy RD, Fogelson AL, Keener JP 2007 Fibrin gel formation in a shear flow,
\textit{Math. Med. Bio.} \textbf{24}, 111--130. 

\bibitem{HLW} Hairer E, Lubich C, Wanner G 2006 Geometric numerical integration, Second Edition,
  Springer Series in Computation Mathematics 31, Springer.

\bibitem{Hammond} Hammond A 2017 Coagulation and diffusion: A probablilistic perspective on the Smoluchowski PDE,
  \textit{Probab. Surveys} \textbf{14}, 205--288.
  
\bibitem{HJR} Harris SC, Johnston SGG, Roberts MI The coalescent structure of continuous-time Galton--Watson trees,
  \textit{Ann. Appl. Prob.} \textbf{30}(3), 1368--1414.
  
\bibitem{H-LOTTW} Henry--Labord\`ere P, Oudjane N, Tan X, Touzi N, Warin X 2018 Branching diffusion representation of semilinear PDEs and Monte Carlo approximation,
  \textit{Ann. Inst. H. Poincar\'e Probab. Statist.} \textbf{55}(1), 184--210.

\bibitem{H-LT} Henry--Labord\`ere P, Touzi N 2018 Branching diffusion representation for nonlinear Cauchy problems and Monte Carlo approximation, \textit{Ann. Appl. Probab.} \textbf{31}(5), 2350--2375.
  
\bibitem{HMP} Hu Y, Mallein B, Pain M 2020 An exactly solvable continuous-time Derrida--Retaux model, \textit{Comm. Math. Phys.} \textbf{375}(1), 605--651.

\bibitem{IM-KNZ} Iserles A, Munthe--Kaas H, N\o rsett SP, Zanna A 2000 Lie-group methods,
\textit{Acta Numer.} \textbf{9}, 215--365.

\bibitem{Ishida} Ishida T, Kawazumi N 2013 The Lie algebra of rooted planar trees, \textit{Hokkaido Mathematical Journal} \textbf{42}, 397--416.

\bibitem{ILP} Iyer G, Leger N, Pego RL 2018 Coagulation and universal scaling limits for critical Galton--Watson processes,
\textit{Advances in Applied Probability} \textbf{50}(2), 504--542.

\bibitem{JKR} Johnston SGG, Kyprianou A, Rogers T 2022 Multitype $\Lambda$-coalescents, arXiv:210314638v2.
  
\bibitem{KPS} Kaganer VM, Ploog KH, Sabelfeld KK 2006 Coarsening of facetted two-dimensonal islands by dynamic coalescence,
\textit{Phys. Rev. B} \textbf{73}, 115425. 

\bibitem{KY} Kaplan JL, Yorke JA 1979 Nonassociative, real algebras and quadratic differential equations,
  \textit{Nonlinear Analysis, Theory, Methods \& Applications} \textbf{3}(1), 49--51.

\bibitem{KB} Keck DD, Bortz DM 2013 Numerical Simulation of solutions and moments of the Smoluchowski coagulation equation,
arXiv:1312.7240v1. 

\bibitem{KP} Kircheis M, Potts D 2019 Direct inversion of the nonequispaced fast Fourier transform, \textit{Linear Algebra and its Applications} \textbf{575}, 106--140.

\bibitem{Krasnov} Krasnov Y 2023 Non-associative structures and their applications in differential equations,
  \textit{Mathematics} \textbf{11}, 1790.

\bibitem{LS} Lambert A, Schertzer E 2020 Coagulation-transport equations and the nested coalescents, \textit{Probab. Theory Relat. Fields} \textbf{176}, 77--147.

\bibitem{LT} Leyvraz F, Tschudi HR 1981 Singularities in the kinetics of coagulation processes, \textit{J. Phys. A: Math. gen.} \textbf{14}, 3389--3405.

\bibitem{LR} Loday J-L, Ronco M 1998 Hopf algebra of the planar binary trees, \textit{Advances in Mathematics} \textbf{139}, 293--309.

\bibitem{LM-Kbackward} Lundervold A, Munthe--Kaas H 2013 Backward error analysis and the substitution law for Lie group integrators,
  \textit{Found. Comput. Math.\/} \text{13}, 161--186.

\bibitem{LM-K} Lundervold A, Munthe--Kaas HZ 2015 On algebraic structures of numerical integration on vector spaces and manifolds,
  In \textit{IRMA Lectures in Mathematics and Theoretical Physics 21}
  on \textit{Fa\`a Di Bruno Hopf Algebras, Dyson--Schwinger Equations, and Lie--Butcher Series}, 219--263.
  
\bibitem{Malham-KdV} Malham SJA 2022 The non-commutative Korteweg--de Vries hierarchy and combinatorial P\"oppe algebra, \textit{Physica D} \textbf{434}, 133228.
  
\bibitem{MalWie} Malham SJA, Wiese A 2009 Stochastic expansions and Hopf algebras,
\textit{Proc. R. Soc. A} \textbf{465}, 3729--3749. (doi:10.1098/rspa.2009.0203)

\bibitem{MarW} Marckert J-F, Wang M 2019 A new combinatorial representation of the additive coalescent, \textit{Random Struct. Alg.} \textbf{54}, 340--370.

\bibitem{Markus} Markus L 1960 Quadratic differential equations and non-associative algebras,
  \textit{Contributions to the theory of nonlinear oscillations}, Eds. L. Cesari, J. LaSalle, S. Lefschetz, Princeton University Press, 185--213. 

\bibitem{McKeanII} McKean HP 1975 Application of Brownian motion to the equation of Kolmogorov--Petrovski--Piskunov,
\textit{Comm. Pure Appl. Math.\/ } \textbf{28}, 323--331.

\bibitem{McLeod1962a} McLeod JB 1962 On an infinite set of non-linear differential equations, \textit{Q. J. Math. Oxford} \textbf{13}(2), 119--128. 

\bibitem{McLeod1962b} McLeod JB 1962 On an infinite set of non-linear differential equations (II), \textit{Q. J. Math. Oxford} \textbf{13}(2), 193--205.

\bibitem{McLeod1962c} McLeod JB 1962 On a recurrence formula in differential equations, \textit{Q. J. Math. Oxford} \textbf{13}(2), 283--284. 
  
\bibitem{MP} Menon G, Pego RL 2004 Approach to self-similarity in Smoluchowski's coagulation equations, 
\textit{Communications on Pure and Applied Mathematics} 
\textbf{LVII}, 1197--1232.

\bibitem{M-KW} Munthe--Kaas H, Wright W 2008 On the Hopf algebraic structure of Lie group integrators,
  \textit{Found. Comput. Math.\/} \textbf{8}, 227--257.

\bibitem{NZ} Normand R, Zambotti L 2011 Uniqueness of post-gelation solutions of a class of coagulation equations, \textit{Ann. I. H. Poincar\'e} -- AN \textbf{28}, 189--215.
  
\bibitem{Pego} Pego RL 2005 Lectures on dynamics in models of coarsening and coagulation,
in \textit{Dynamics in models of coarsening, coagulation, condensation and quantization},
Editors W. Bao and J.G. Liu, World Scientific. 
https://doi.org/10.1142/9789812770226\_0001

\bibitem{PTVF} Press WH, Teukolsky SA, Vetterling WT, Flannery BP 1997 \textit{Numerical recipes in C: The art of scientific computing}, 2nd Ed., CUP. 

\bibitem{PS} Pressley A, Segal G 1986 \emph{Loop groups}, Oxford Mathematical Monographs, Clarendon Press, Oxford.

\bibitem{Reutenauer} Reutenauer C 1993 \textit{Free Lie algebras}, LMS Monographs New Series 7, Clarendon Press, Oxford.

\bibitem{Rohrl} R\"ohrl H 1977 Algebras and differential equations, \textit{Nagoya Math. J.} \textbf{68}, 59--122.

\bibitem{SPI} Samsel RW, Perelson AS 1982 Kinetics of Rouleau formation I: A mass action approach with geometric features,
\textit{BioPhys. J.} \textbf{37}, 493--514.

\bibitem{SPII} Samsel RW, Perelson AS 1984 Kinetics of Rouleau formation II: Reversible reactions,
\textit{BioPhys. J.} \textbf{45}, 805--824.

\bibitem{Scott} Scott WT 1968 Analytic studies of cloud droplet coalescence I,
\textit{Journal of Atmospheric Sciences} \textbf{25}, 54--65.

\bibitem{ShethPitman} Sheth RK, Pitman J 1997 Coagulation and branching process models of gravitational clustering,
  \textit{Mon. Not. R. Astron. Soc.\/ } \textbf{289}, 66--82.

\bibitem{Spouge} Spouge JL 1985 Analytical solutions to Smoluchowski's coagulation equation: a combinatorial interpretation,
  \textit{J. Phys. A: Math. gen.} \textbf{18}, 3063--3069.
  
\bibitem{Stanley} Stanley RP 1999 \textit{Enumerative Combinatorics}, Volume 2, Cambridge Studies in Advanced Mathematics 62, CUP.

\bibitem{SJTCE} Stoldt CR, Jenks CJ, Thiel PA, Cadilhe AM, Evans JW 1999
Smoluchowski ripening of AG islands on AG(100), \textit{J. Chem. Phys.} \textbf{111}(11), 5157--5166.

\bibitem{Throm} Throm S 2023 Uniqueness of measure solutions for multi-component coagulation equations, arXiv:2303.00775v2.

\bibitem{vRS} van Roessel HJ, Shirvani M 2006 A formula for the post-gelation mass of a coagulation equation with a separable bilinear kernel,
  \textit{Physica D} \textbf{222}, 29--36.
  
\bibitem{WWLGF} Winkler K, Wojciechowski T, Liszewska M, G\'orecka E, Fialkowski M 2014
Morphological changes of Gold nanoparticles due to adsorption onto Silicon substrate and Oxygen plasma treatment, 
\textit{RSC Adv.} \textbf{4}, 12729.

\bibitem{WPEA} Woehl TJ, Park C, Evans JE, Arslan I, Ristenpart WD, Browning ND 2014
Direct obsservation of aggregative nanoparticle growth: Kinetic modeling of size distribution and growth rate,
\textit{Nano Lett.} \textbf{14}, 373--378.

\bibitem{ZKR} Zidar M, Kuzman D, Ravnik M 2018
Characterisation of protein aggregation with the Smoluchowski coagulation approach for use in biopharmaceuticals,
\textit{Soft Matter} \textbf{14}, 6001.
  
\end{thebibliography}
\end{document}